\documentclass[11pt]{amsart}
\usepackage{geometry}                
\usepackage{graphicx}
\usepackage{amsmath, amsthm, amssymb}
\usepackage{epstopdf}
\usepackage[pagewise]{lineno}

\newcommand\R{\mathbb R}
\newcommand\C{{\mathbb C}}

\newcommand\g{{\mathfrak g}}
\renewcommand\k{{\mathfrak k}}
\newcommand\m{{\mathfrak m}}
\renewcommand\t{{\mathfrak t}}
\newcommand\q{{\mathfrak q}}
\renewcommand\a{{\mathfrak a}}
\newcommand\h{{\mathfrak h}}

\newcommand\quot[2]{#1/\!\!/#2}
\newcommand\Ad{{Ad}}
\newcommand\Exp{{Exp}}
\renewcommand\span{{span}}

\renewcommand\O{{\mathcal O}}
\newcommand\Ou{{\mathcal O_u}}
\newcommand\Onc{{\mathcal O_{0}}}

\newcommand\ad{{ad}}
\newcommand\twist[3]{{#1*^{#2}{#3}}}
\newcommand\proj[2]{{proj^{#1}_{#2}}}
\newtheorem{lemma}{Lemma}[section]
\newtheorem{defn}[lemma]{Definition}
\newtheorem{prop}[lemma]{Proposition}
\newtheorem{theorem}[lemma]{Theorem}
\newtheorem{cor}[lemma]{Corollary}

\begin{document}

\author{Ralph J. Bremigan}
\address{Department of Mathematical Sciences, Ball State University, Muncie IN 47306-0490}
\email{bremigan@bsu.edu}

\title[Complexified Hermitian Symmetric Spaces]{Complexified Hermitian Symmetric Spaces, \\Hyperk\"ahler Structures, and Real Group Actions}


\begin{abstract}
There is a known hyperk\"ahler structure on any complexified Hermitian symmetric space $G/K$, whose construction relies on identifying $G/K$ with both a (co)adjoint orbit  and the cotangent bundle to the compact Hermitian symmetric space $G_u/K_0$.  Via a family of explicit diffeomorphisms, we show that almost all of the complex structures are equivalent to the one on $G/K$; via a family of related  diffeomorphisms, we show that almost all of the symplectic structures are equivalent to the one on $T^*\left(G_u/K_0\right)$.  We highlight the intermediate K\"ahler structures, which share a holomorphic action of $G$ related to the one on $G/K$, but moment geometry related to that of $T^*\left(G_u/K_0\right)$.  As an application, for the real form $G_0\subset G$ corresponding to $G_0/K_0$,  the Hermitian symmetric space of noncompact type, we give a strategy for study of the action on $G/K$ using the moment-critical subsets for the intermediate structures.  We give explicit computations for $SL(2)$.  
\end{abstract}
\thanks{Primary  53C26;  Secondary 32M15, 53C35, 53D20}
\date{July 23, 2019}

\maketitle

\section{Introduction}\label{s:intro}

This paper serves two purposes:

\begin{itemize}
\item To extend understanding of the hyperk\"ahler structure on complexified Hermitian symmetric spaces.  
\item To begin an exploration of moment-map techniques in this specific hyperk\"ahler context.  
\end{itemize}

We begin by addressing the first of these purposes.  The hyperk\"ahler structure has been developed by several authors, including D. Burns \cite{Bu}, A. Dancer and R. Sz\"oke \cite{DS}, and culminating in the papers of O. Biquard and P. Gauduchon \cite{BG2}, \cite{BG3}, \cite{BG4}.   (See Section \ref{s:notation} below for notation and basic facts about Hermitian symmetric spaces, and Section \ref{s:notation2} for an overview of hyperk\"ahler structures.)  The structure relies on two  realizations of the complexified Hermitian symmetric spaces, and structures on them, that are linked in a nontrivial way.  One such realization is as the adjoint orbit of an element $\Upsilon\in\g$ by a complex semisimple group $G$.  On this space, there is an obvious complex structure, which we denote $J_1$, and the Kirillov-Kostant-Souriau holomorphic symplectic form, which we denote $-\omega_3+i\omega_2$ (breaking the complex-valued form into its real and imaginary parts).  The second such realization of the complexified Hermitian symmetric space is as the cotangent bundle to a flag manifold $G/Q$; this gives a complex structure $J_3$, as well as the usual form on a cotangent bundle.  We can write this form as $\omega_1+i\omega_2$ (where, crucially,  the two occurrences  of $\omega_2$ coincide).  The flag manifold $G/Q$ can be viewed as the adjoint orbit of $\Upsilon$ by a compact real form $G_u$ of $G$.  The bridge between these two realizations (recalled in Section \ref{s:geometry}) is a $G_u$-equivariant map that allows us to consider $J_1$ and $\omega_3$ in the second realization and $J_3$ and $\omega_1$ in the first realization.  Thus in both realizations, we can see  the totality of the hyperk\"ahler structure (complex structures 
$J_1$, $J_3$, $J_2:=J_3J_1$; real symplectic forms $\omega_1$, $\omega_2$, $\omega_3$; metric $m$).  

A feature of every hyperk\"ahler structure is the existence, not only of $J_1$, $J_2$, and $J_3$, but of a complex structure for each unit linear combination $J_{(\lambda_1,\lambda_2,\lambda_3)}:=\lambda_1J_1+\lambda_2J_2+\lambda_3J_3$, which has an associated symplectic form $\omega_{(\lambda_1,\lambda_2,\lambda_3)}:=\lambda_1\omega_1+\lambda_2\omega_2+\lambda_3\omega_3$.  We can identify the triple $(\lambda_1,\lambda_2,\lambda_3)$ with a point $\lambda$ in the extended complex plane $\C\cup\{\infty\}$ via stereographic projection (Definition \ref{d:stereo}).  Remarkably, the complex number $\lambda$ plays a role in an explicit $G_u$-equivariant self-bijection $T_\lambda$ of the complexified Hermitian symmetric space, which provides a new understanding of how the complex structure $J_{(\lambda_1,\lambda_2,\lambda_3)}$ and the symplectic form $\omega_{(\lambda_1,\lambda_2,\lambda_3)}$ vary with $(\lambda_1,\lambda_2,\lambda_3)$.  The main results are in Section \ref{s:deformations}.  Among them:

\begin{itemize}
\item Via the map $T_\lambda$, each $J_{(\lambda_1,\lambda_2,\lambda_3)}$ except $\pm J_3$ is equivalent 
to $J_1$   (Theorem \ref{t:equiv}.)  
\item Via a map related to $T_\lambda$, for any $(\lambda_1,\lambda_2,\lambda_3)$ with $\lambda_3\neq 0$, the symplectic form $\omega_{(\lambda_1,\lambda_2,\lambda_3)}$ is equivalent to a scalar multiple of $\omega_3$  (Theorem \ref{t:J3}).
\end{itemize}
Thus, from the point of view of complex structures, we can regard $J_1$ and the $G$-orbit of $\Upsilon$ as a model for almost all the complex structures and for the associated holomorphic actions of $G$, with a degeneracy at $\pm J_3$.  From the point of view of real symplectic geometry and associated moment maps, we can regard the cotangent bundle to the $G_u$ orbit of $\Upsilon$ as a model, with degeneracy along the circle of complex structures that contains $J_1$ and $J_2$.  

We now turn briefly to the second purpose of this paper.  There is a theory of  the  use of  moment-map techniques to investigate actions of complex reductive groups, notably developed in the  work of F. Kirwan (\cite{Ki}; see also \cite{N}, \cite{KN}).  This theory has received a significant extension  to certain actions of real reductive groups, by P. Heinzner, G. Schwarz, and H. St\"otzel \cite{HS}, \cite{HHS}.  Their techniques apply here to the study of the action of $G_0$ (the real form of $G$ corresponding to the Hermitian symmetric space of noncompact type) on the complexified Hermitian symmetric space.  In fact, there is a  family of actions of $G_0$, the family being parametrized by a 2-sphere, each action corresponding to a choice of  $(\lambda_1,\lambda_2,\lambda_3)$, with its complex structure, symplectic form, and moment map.  Here, our main results shed light, both qualitatively and computationally, on these actions:  we have that most actions are equivalent to the usual one on the $G$-orbit of $\Upsilon$; most moment-critical subsets are equivalent to the ones for the action on the cotangent bundle to $G/Q$; and the maps $T_\lambda$ allow us to make explicit computations.  This points out the singular nature of both of our initial realizations of the complexified Hermitian symmetric space (the ones associated to $J_1$ and $J_3$), since for most choices of complex structure (a choice of $(\lambda_1,\lambda_2,\lambda_3)$ with $\lambda_3\notin\{0,1,-1\}$), the action is equivalent to the one associated to $J_1$ (but not $J_3$), and the symplectic geometry and moment-critical subsets are equivalent to the ones associated to $J_3$ (but not $J_1$).  It is not hard to imagine that study of these intermediate structures would lead to new perspectives on the more familiar $J_1$ and $J_3$ structures, e.g., in the construction of representations.  A full working out for arbitrary Hermitian symmetric spaces is beyond the scope of this paper, but in Section \ref{s:example}, we present a detailed account (without proofs) for $G=SL(2,\C)$.

It should be noted that hyperk\"ahler structures on coadjoint orbits are known to exist outside of the setting of complexified Hermitian symmetric spaces, and it would be interesting to know how the results of this paper extend.  However, only in the Hermitian symmetric setting is the hyperk\"ahler structure known so explicitly.

We conclude this introduction with a few words about the structure of this paper, and our computational strategy and proofs (both of the new results and the exposition  of the hyperk\"ahler structure).  The formulas for $\omega_1$, $\omega_2$, $\omega_3$, and the metric $m$ depend on specific choices of tangent vectors, and  are expressed as modifications of the Killing form.  We have expressed these modifications in terms of operators that are written as power series, and which ultimately derive from the exponential map.  These operators are defined in Section \ref{s:computation}, where we prove some basic properties; in the same section, we define the operators $T_\lambda$, which are so important in understanding the deformation results stated in Section \ref{s:deformations}.  Our version of the development of the hyperk\"ahler structure on complexified Hermitian symmetric spaces is carried out in Section \ref{s:quaternionic} and the first part of Section \ref{s:hyper}.  Most of this is straightforward, using the results developed in Section \ref{s:computation}.  The second half of Section \ref{s:hyper} includes some additional information (e.g., about K\"ahler potentials) that is not needed elsewhere.  In Section \ref{s:moment}, we recall and prove known formulas for the moment maps; we expect to use this in future work, but it is not needed elsewhere in the paper.  Our main results about deformations are stated in Section \ref{s:deformations} and proved in Section \ref{s:proofs}.  The technical results from Section \ref{s:computation} are used, along with some minor bookkeeping involving the Killing form.  The more substantive arguments in Section \ref{s:proofs} are made necessary by an essential but complicating property of the map $T_\lambda$:  it relies on a map $t_\lambda:\m\rightarrow \m$ that does not generally preserve the real form $\m_0$.  Since the description of the hyperk\"ahler structure relies on a decomposition of the complexified Hermitian symmetric space that involves $\m_0$ (Section \ref{sss:OandOu}), the map $T_\lambda$ is a bit of challenge.  The more complicated arguments in Section \ref{s:proofs}, at their heart,  contend with the problem of resolving tangent vectors into parts corresponding to $\m_0$ and $i\m_0$.

\section{Notation and recap of Hermitian symmetric spaces}\label{s:notation}

See \cite{D}, \cite{W}.

German letters denote Lie algebras of the corresponding Lie groups. 

We use $I$ to denote the identity map, and $J$ (with various subscripts) to denote  complex structures.

Let $G$ be a simple  algebraic group  over $\mathbb C$, endowed with commuting antiholomorphic involutions  $\sigma$ and $\theta$.  Assume that the fixed point group $G_u:=G^\theta$ is a maximal compact subgroup of $G$. 
The fixed point group  $G_0:=G^\sigma$  is another real form of $G$. We use $\sigma$ and $\theta$ to denote the involutions of the Lie algebra $\g$ as well. 

Let $K$ denote the fixed points of the algebraic involution $\sigma\theta$. Note that $K$ is a complex reductive group, and  $K_0:=K\cap G_0$ is both a compact real form of $K$ and a maximal compact subgroup of $G_0$.  
We use $\Re$ and $\Im$ to denote the real and imaginary parts with respect to the decomposition $\g=\g_u\oplus i\g_u$ (as well as for the real and imaginary parts of complex numbers).   

We have Lie algebra decompositions $\g_0=\k_0\oplus \m_0$ and $\g_u=\k_0\oplus \m_u$ (with $\m_u=i\m_0$); these are  $\pm 1$ eigenspaces decompositions for $\theta$ and $\sigma$, respectively, and are orthogonal decompositions with respect to the Killing form.  By complexifying, we have $\k=\C\otimes_\R \k_0$ and $\m=\C\otimes_\R\m_0$, and then $\g=\k\oplus\m$ is the $\pm1$ eigenspace decomposition of $\g$ with respect to $\sigma\theta$.  

We assume thoughout that $G_u/K_0$ is an \underbar{Hermitian symmetric space}  (irreducible and of compact type); it follows that $G_0/K_0$ is an Hermitian symmetric space of noncompact type.  There is a unique (up to sign) element $\Upsilon\in\k_0$ in the center of $\k_0$  such that the eigenvalues of $\ad_\Upsilon$ on $\m$ are $\pm i$. We write $\m=\m^+\oplus \m^-$ for the eigenspace decomposition, noting that 
$\m^+$ and $\m^-$ are complex vector spaces.  We let $J:\m\rightarrow\m$ denote the operator $\ad_\Upsilon$; the restriction of $J$ to $\m_0$ (resp.  $\m_u$) gives the almost complex structure on $G_0/K_0$ (resp. $G_u/K_0$) at the base point.

Using the definition of $\Upsilon$, we have that $\sigma(\m^+)=\m^-=\theta(\m^+)$, 
and likewise, $\sigma$ and $\theta$ send $\m^-$ to $\m^+$.  It follows that every 
element $X$ of $\m_0$ can be written  as $(I+\sigma)X^+=(I-\theta)X^+$ for a unique 
$X^+\in\m^+$, and each element of $\m_u$ is $(I+\theta)X^+=(I-\sigma)X^+$ for a unique $X^+\in\m^+$. 
Note that $\ad_\Upsilon((I+\sigma)(X^+))=(I+\sigma)(iX^+)$, with the analogous statement when $\sigma$ is replaced with $\theta$.  Thus the $\R$-linear map $\m^+\rightarrow \m_0$, $X^+\mapsto (I+\sigma)X^+$  intertwines the almost complex structures $i$ and $J$, with analogous statement for  $\m_u$ and $\theta$.  

Since $\theta$ is a Lie algebra automorphism of $\g_0$, we have that $[\k_0,\k_0]\subseteq \k_0$, 
$[\k_0,\m_0]\subseteq\m_0$, and $[\m_0,\m_0]\subseteq\k_0$, with similar statements for $\g$, $\k$, and $\m$.
From the fact that $\Upsilon$ is in the center of $\k$, it follows easily that $[\k,\m^+]\subseteq\m^+$ and
$[\k,\m^-]\subseteq\m^-$.   The restriction of the Killing form to either $\m^+$ or $\m^-$ is identically zero.  
Since the bracket of any two elements of $\m^+$ would be an eigenvector of $\ad_\Upsilon$ with eigenvalue $2i$, and since this eigenvalue does not occur, we have that $[\m^+,\m^+]=\{0\}$, and likewise $[\m^-,\m^-]=\{0\}$.  
If $Z\in\m_0$, then the action of $\ad_Z$ on $\g_0$ is diagonalizable with real eigenvalues, and the action of $\ad_Z^2$ on $\m_0$ is diagonalizable with nonnegative real eigenvalues.

Any Cartan subalgebra  $\t\subseteq \k$ contains $\Upsilon$; moreover $\t$ is a Cartan subalgebra of $\g$ (since any element of $\g$ commuting with $\t$ would commute with $\Upsilon$, but 
$\ad_\Upsilon$ has eigenvalues $\pm i$ on $\m$). 
The Cartan subalgebra $\t$ is the Lie algebra of a maximal complex torus $T$ of $K$ and $G$. 
We fix a Cartan subalgebra  $\t\subset \k$ that is stable under $\theta$ (such exist, and all such are conjugate under $K_0$).  
Since $\Upsilon\in\t$, we see that each root space $\g_{\alpha}$ is contained in $\k$, $\m^+$, or $\m^-$.  We may take
 an ordering of the roots $\Phi(\g,\t)$ so that $\m^+$ is contained in the sum of the positive root spaces.  

Given $\alpha\in\Phi(\m^+,\t)$, one can choose $e_\alpha\in\g_\alpha$ such that 
$\left(h_\alpha:=[e_\alpha, \sigma e_\alpha], e_\alpha, f_\alpha:=\sigma e_\alpha\right)$ forms a standard split $sl_2$-triple, meaning $[h_\alpha,e_\alpha]=2e_\alpha$,  $[h_\alpha, f_\alpha]=-2f_\alpha$, and $[e_\alpha, f_\alpha]=h_\alpha$.  
Note that $h_\alpha\in i\k_0\cap \t$, $e_\alpha+f_\alpha\in\m_0$, and $e_\alpha-f_\alpha\in \m_u$.   Define $x_\alpha=e_\alpha+f_\alpha$ and 
$y_\alpha=i(e_\alpha-f_\alpha)$.  With regard to complex structures, we have that $\ad_\Upsilon x_\alpha=y_\alpha$ and $\ad_\Upsilon y_\alpha=-x_\alpha$.  
We define $\Upsilon_\alpha=\frac i2h_\alpha$.  We have the commutation relations
$[\Upsilon_\alpha,x_\alpha]=y_\alpha$, $[\Upsilon_\alpha,y_\alpha]=-x_\alpha$, and $[x_\alpha,y_\alpha]=-4\Upsilon_\alpha$.\footnote{A good example to keep in mind is $G_u=SU(2)$, $G_0=SU(1,1)$. 
For any $C\in\C$ with $|C|=1$, we can take  $e_\alpha=\left(\begin{array}{cc}0&C\\0&0\end{array}\right)$, $f_\alpha=\left(\begin{array}{cc}0&0\\\overline{C}&0\end{array}\right)$, and $h_\alpha=\left(\begin{array}{cc}1&0\\0&-1\end{array}\right)$.
Then $x_\alpha=\left(\begin{array}{cc}0&C\\\overline C&0\end{array}\right)$,  $y_\alpha=\left(\begin{array}{cc}0&iC\\-i\overline C&0\end{array}\right)$, and 
$\Upsilon_\alpha=
\frac i2\left(\begin{array}{cc}1&0\\0&-1\end{array}\right)$.  For the choice $C=1$, we would have  $e_\alpha=\left(\begin{array}{cc}0&1\\0&0\end{array}\right)$, $f_\alpha=\left(\begin{array}{cc}0&0\\1&0\end{array}\right)$,  $h_\alpha=\left(\begin{array}{cc}1&0\\0&-1\end{array}\right)$, 
$x_\alpha=\left(\begin{array}{cc}0&1\\1&0\end{array}\right)$,  $y_\alpha=\left(\begin{array}{cc}0&i\\-i&0\end{array}\right)$, and $\Upsilon_\alpha=
\frac i2\left(\begin{array}{cc}1&0\\0&-1\end{array}\right)$. With regard to the Cayley transform, we have $c=\frac 1{\sqrt 2}\left(\begin{array}{cc}1&-1\\1&1\end{array}\right)$ and $c^2=\left(\begin{array}{cc}0&-1\\1&0\end{array}\right)$.}
We obtain a real basis of $\m_0$ from the elements $x_\alpha$ and 
$y_\alpha$, where $\alpha$ runs through the positive $\m$-roots.

There exists a  {maximal set $\Psi=\{\psi_1,\dots,\psi_r\}$ of strongly orthogonal $\m^+$-roots} (meaning that for $i\neq j$, $\psi_i\pm\psi_j$ are not roots).  
For any $\psi_i$ and $\psi_j$ in $\Psi$ with $i\neq j$, each element in 
$\{e_{\psi_i}, f_{\psi_i}, h_{\psi_i}, x_{\psi_i}, y_{\psi_i}, \Upsilon_{\psi_i}\}$ commutes with each element in 
$\{e_{\psi_j}, f_{\psi_j}, h_{\psi_j}, x_{\psi_j}, y_{\psi_j}, \Upsilon_{\psi_j}\}$.  
Otherwise said: let $\g[\psi]=span_\C\{e_{\psi},f_{\psi},h_{\psi}\}\simeq sl(2,\C)$; then for $i\neq j$, each element of $\g[\psi_i]$ commutes with each element of $\g[\psi_j]$. 
Let  $\a_0=\span_\R \{x_\psi:\psi\in\Psi\}$, an  abelian subalgebra of $\g_0$, contained in $\m_0$, and maximal with respect to these properties; let $\a=\C\otimes_{\R}\a_0$.  Any two maximal abelian subalgebras of $\m_0$ are $K_0$-conjugate, and every element of $\m_0$ lives in some such maximal abelian subalgebra.   
  We  define $\Upsilon'\in\t\cap\k_0$ by the formula $\Upsilon=\Upsilon'+\sum_{\psi\in\Psi}\Upsilon_\psi$; note that  $\Upsilon'$ commutes with each $\g[\psi]$.

Let $\t^-=\sum_{\psi\in\Psi}\C h_\psi$, and let  $\t^+=\{h\in\t:\psi(h)=0$  for all $\psi\in\Psi\}$, the orthogonal complement in $\t$ of $\t^-$.  Elements of $\Phi(\g,\t^-)$ are called {\em restricted roots.}  Results of Harish-Chandra and Moore give descriptions of restricted roots and their multiplicities; in particular, the only possible restricted roots are the restrictions to $\t^-$ of $\pm \frac 12 \psi_s\pm\frac 12\psi_t$   and $\pm\frac 12\psi_t$ for roots $\psi_s,\psi_t\in\Psi$.  
Define $c\in G_u$ by $c=\exp\left(\sum_{\psi\in\Psi}\frac{i\pi}4y_\psi\right)$, a  {\em (full) Cayley transform}.  Under the adjoint representation, $c$ sends $h_\psi$ to $x_\psi$, $x_\psi$ to $-h_\psi$, and $y_\psi$ to itself. 
Note that  $\Ad_c$ centralizes $\t^+$ and conjugates $\t^-$ to $\a$.  
Let $\h=\Ad_c\t=\t^+\oplus \a$, another Cartan subalgebra of $\g$.  The map $\Ad_c$ allows us to identify  the roots systems $\Phi(\g,\t)$ and $\Phi(\g,\h)$.  Elements of $\Phi(\g,\a)$ are again called restricted roots.  The Weyl group $N_{K_0}(\a_0)/Z_{K_0}(\a_0)$ acts on $\{\pm x_{\psi}\}_{\psi\in\Psi}$ as signed permutations.  

Let $\q=\k\oplus\m^-$; it is the Lie algebra of a parabolic subgroup $Q\subseteq G$.  It is well-known that 
$G_u\cap Q=K_0$ and that $G_u$ acts transitively on $G/Q$.  Moreover, if we let $\Ou\subset \g_u
$ denote the $\Ad(G_u)$-orbit of $\Upsilon$, then the $G_u$-stabilizer of 
$\Upsilon$ is $K_0$.  We obtain the identifications 
$\Ou:=\Ad_{G_u}(\Upsilon)\simeq G_u/K_0\simeq G/Q.$
There is a corresponding realization of the noncompact Hermitian symmetric space  as 
$G_0/K_0\simeq \Onc\subseteq \g_0$, where $\Onc$ is the $\Ad(G_0)$-orbit of $\Upsilon$.  
We define  $\O\subset\g$, the $\Ad(G)$-orbit of $\Upsilon$; here $\O\simeq G/K$. 
The isomorphism of $G_u/K_0$ with   $G/Q$, and the inclusion of $G_0/K_0$ in $G/Q$,  give  complex structures on $\Ou$ and $\Onc$;  these agree with the almost complex structures defined by $\Upsilon$.     The space $\O$
is  the ``complexified Hermitian symmetric space'' in the title of this paper.

\section{Hyperk\"ahler structures}\label{s:notation2}

See \cite{BG4}, \cite{H}.  

We recall that a hyperk\"ahler structure on a real manifold $M$ encompasses the following features:

\begin{itemize}
\item Almost complex structures $J_1$, $J_2$, $J_3$ on $M$  that obey the usual quaternionic multiplication rules.  It follows that for each unit vector $(\lambda_1,\lambda_2, \lambda_3)\in\R^3$, the operator $J_{(\lambda_1,\lambda_2,\lambda_3)}:=\lambda_1J_1+\lambda_2J_2+\lambda_3J_3$ gives an almost complex structure.  
\item Integrable complex structures underlying the above almost complex structures.  
\item A Riemannian metric $m$ on $M$ that is invariant under each $J_{(\lambda_1,\lambda_2,\lambda_3)}$.  
\item For each unit vector ${(\lambda_1,\lambda_2,\lambda_3)}$, a real, closed, nondegenerate, $J_{(\lambda_1,\lambda_2,\lambda_3)}$-invariant 2-form $\omega_{{(\lambda_1,\lambda_2,\lambda_3)}}$ on $M$.  This form is given explicitly by the rule 
$\omega_{(\lambda_1,\lambda_2,\lambda_3)}(v,w)=-m(v,J_{(\lambda_1,\lambda_2,\lambda_3)} w)$.  
If $\omega_1$, $\omega_2$, $\omega_3$ denote the 2-forms associated to $J_1$, $J_2$, $J_3$, then $\omega_{{(\lambda_1,\lambda_2,\lambda_3)}}=\sum_i\lambda_i\omega_i$.  
\item  For each $J_{(\lambda_1,\lambda_2,\lambda_3)}$, a class of nondegenerate holomorphic symplectic forms (where ``class'' means up to scalar multiples by $\C^*$).  For $J_{(\lambda_1,\lambda_2,\lambda_3)}$ fixed, one such holomorphic symplectic form is obtained via unit vectors ${(\lambda_1',\lambda_2',\lambda_3')}$ and ${(\lambda_1'',\lambda_2'',\lambda_3'')}$, where 
${(\lambda_1,\lambda_2,\lambda_3)}$, ${(\lambda_1',\lambda_2',\lambda_3')}$, and ${(\lambda_1'',\lambda_2'',\lambda_3'')}$ are mutually orthogonal and obey the right-hand rule.  The associated holomorphic symplectic form is 
$\omega_{{(\lambda_1',\lambda_2',\lambda_3')}}   +   i\,\omega_{{(\lambda_1'',\lambda_2'',\lambda_3'')}}$.
\end{itemize}

There is redundancy among these conditions.  In Proposition \ref{prop:existence}, we use a smaller set of conditions that determine the hyperk\"ahler structure.

Our main results will rely on identifying the complex structures $\{J_{(\lambda_1,\lambda_2,\lambda_3)}\}$ with points in the extended complex plane:  

\begin{defn}\label{d:stereo}
For $\lambda\in\C\cup\{\infty\}$, let $(\lambda_1,\lambda_2,\lambda_3)$ denote the point on the unit 2-sphere in $\R^3$ defined by stereographic projection from the south pole:

$$\lambda_1=\frac{\lambda+\overline\lambda}{1+|\lambda|^2}\qquad\qquad \lambda_2=\frac{\lambda-\overline\lambda}{i(1+|\lambda|^2)}\qquad\qquad \lambda_3=\frac{1-|\lambda|^2}{1+|\lambda|^2}$$
 
If $\lambda\in\C\cup\{\infty\}$, let $J_{(\lambda)}$ denote $J_{(\lambda_1,\lambda_2,\lambda_3)}$, with $\lambda_1,\lambda_2,\lambda_3$ as above.  Similarly, we use $\omega_{(\lambda_1,\lambda_2,\lambda_3)}$ and $\omega_{(\lambda)}$ interchangeably.  \end{defn}

Note that for $\lambda\in\C^*$ exactly when $J_{(\lambda)}$ is not $\pm J_3$, as $J_3=J_{(0)}$ and $-J_3=J_{(\infty)}$.  Note moreover that $J_1=J_{(1)}$ and $J_2=J_{(i)}$.

\section{Computational preliminaries}\label{s:computation}

In this section, we define and study some operators that are essential for expressing the hyperk\"ahler structure on $\O$.  

$\bullet$ In \S\ref{ss:EZFZ}: For $Z\in\m_0$, the linear map $E_Z:\m\rightarrow\m, \m_0\rightarrow\m_0$ is the power series in $\ad_Z$ associated to the differential of the exponential map $\m_0\rightarrow G_0/K_0$, or equally well, the even part of the differential of the exponential map $\g\rightarrow G$.  

$\bullet$ In \S\ref{ss:SZ}:  For $Z\in\m_0$, the linear map $S_Z:\m\rightarrow\m, \m_0\rightarrow\m_0$ is the power series in $\ad_Z$ that gives $W\mapsto proj^{\g}_{\m}\Ad_{\exp Z}W$. 

These power series have only even terms.  In Proposition \ref{prop:squarescommute} and Corollary \ref{cor:order2again} (\S\ref{ss:powerseries}), we verify the crucial fact that such power series commute with their $J$-conjugates.  

$\bullet$ In \S\ref{ss:P}:  The real analytic map $P:\m_0\rightarrow\m_0$ is closely related to the map  $Z\mapsto proj^{\g}_{\m}\left(\Ad_{\exp Z}\Upsilon\right)$.  

$\bullet$ In \S\ref{ss:t}: The maps $t_\lambda:\g\rightarrow\g$ and $T_\lambda:\O\rightarrow\O$ are crucial for stating the deformation results in \S\ref{s:deformations}.

\subsection{Power series and $\ad_{Z}^2$}\label{ss:powerseries}

Let $p(z)$ denote a power series with complex coefficients, let $Z=Z^++Z^-\in\m$, and suppose that $p(\ad_Z)$ converges as an operator on $\g$.  We write $p(z)=p_{even}(z)+p_{odd}(z)$ to represent the decomposition into sums of even and odd powers of $z$.  
It is trivial  that $p_{even}(\ad_Z)$ preserves $\k$ and $\m$ while 
$p_{odd}(\ad_Z)$ maps $\k$ to $\m$ and $\m$ to $\k$.  If $Z\in\g_0$ (resp. $\g_u$) and if the power series has real coefficients, then these statements also hold with subscripts of $0$ (resp. $u$) inserted.  Recall that $\m_0=i\m_u$; throughout, we  use the fact that if $Z\in\m_0$ and $p$ has real coefficients, then $p_{even}(\ad_Z)$ preserves both $\m_0$ and $\m_u$. 

Simple computations underlie the following lemma: 

\begin{lemma}\label{lemma:order2}       

Let $Z,W\in\m$ and $X\in\k$. 
\begin{enumerate}
\item $\ad_{JZ}^2(W)=-J\left(\ad_Z^2(JW)\right)=-\ad_{Z^-}^2W^++\ad_{Z^-}\ad_{Z^+}W^-+\ad_{Z^+}\ad_{Z^-}W^+-\ad_{Z^+}^2W^-$.
\item If $p$ only has even-order terms, then as endomorphisms of $\m$, we have $p\left(\ad_{JZ}\right)=-J\circ p\left(\ad_Z\right)\circ 
J$. 
\item $[JZ,JW]=[Z,W]$ and $J[X,W]=[X,JW]$.  \qed
\end{enumerate}
\end{lemma}

\begin{prop}\label{prop:squarescommute}

For all  $Z\in\m$,  $\ad_{JZ}^2$ and  $\ad_{Z}^2$ commute as endomorphisms of $\m$. 
\end{prop}

\begin{proof}  
We repeatedly will use that $\ad_{[A,B]}=\ad_A\ad_B-\ad_B\ad_A$.  

Write $Z=Z^++Z^-$ and $JZ=i(Z^+-Z^-)$.   Substituting in the equation $\ad_Z^2\ad_{JZ}^2=\ad_{JZ}^2\ad_Z^2$, expanding, and eliminating terms with the zero-operators $\ad_{Z^+}^3$ or $\ad_{Z^-}^3$, we see that the proposition is equivalent to $0=\ad_{Z^+}^2\ad_{Z^-}\ad_{Z^+}-\ad_{Z^+}\ad_{Z^-}\ad_{Z^+}^2$ and $0=\ad_{Z^-}^2\ad_{Z^+}\ad_{Z^-}-\ad_{Z^-}\ad_{Z^+}\ad_{Z^-}^2$, and invoking symmetry, it is sufficient to prove the second equation.  The operator $\ad_{Z^-}^2\ad_{Z^+}\ad_{Z^-}-\ad_{Z^-}\ad_{Z^+}\ad_{Z^-}^2$ is zero on $\k\oplus\m^-$; we must prove that it is zero as a function from $\m^+$ to $\m^-$, which is to say, $\ad_{Z^-}\ad_{[Z^-,Z^+]}\ad_{Z^-}W^+=0$ for all $W^+\in \m^+$.  

We begin with $0=\ad_{Z^-}\ad_Z^+W^+$, valid since $\m^+$ is abelian.  Rewrite this equation as $0=\ad_{[Z^-,Z^+]}W^++\ad_{Z^+}\ad_{Z^-}W^+$.  Next apply $\ad_{Z^-}$; after using the Jacobi identity, we obtain the equation
\begin{equation} 0=\ad_{[Z^-,[Z^-,Z^+]]}W^++\ad_{Z^-}\ad_{Z^+}\ad_{Z^-}W^++\ad_{[Z^-,Z^+]}\ad_{Z^-}W^+\end{equation}  Again apply $\ad_{Z^-}$, yielding the equation
\begin{equation}0=\ad_{Z^-}\ad_{[Z^-,[Z^-,Z^+]]}W^++\ad_{Z^-}^2\ad_{Z^+}\ad_{Z^-}W^++\ad_{Z^-}\ad_{[Z^-,Z^+]}\ad_{Z^-}W^+\label{eq:ads}\end{equation} 
The first term on the righthand side is equal to $\ad_{[Z^-,[Z^-,Z^+]]}\ad_{Z^-}W^+$ (since $\ad_{Z^-}^3$ is the zero-operator) and can be rewritten as $\ad_{Z^-}\ad_{[Z^-,Z^+]}\ad_{Z^-}W^+-\ad_{[Z^-,Z^+]}\ad_{Z^-}^2W^+$.  Thus the right-hand side of \eqref{eq:ads} is \begin{equation}2\ad_{Z^-}\ad_{[Z^-,Z^+]}\ad_{Z^-}W^++\ad_{Z^-}^2\ad_{Z^+}\ad_{Z^-}W^+-\ad_{[Z^-,Z^+]}\ad_{Z^-}^2W^+\label{eq:ads2}\end{equation}  
The last two terms in \eqref{eq:ads2} sum to \begin{equation*} \ad_{Z^-}\ad_{Z^-}\ad_{Z^+}\ad_{Z^-}W^+-\ad_{Z^-}\ad_Z^+\ad_{Z^-}^2W^++\ad_{Z^+}\ad_{Z^-}^3W^+=\ad_{Z^-}\ad_{[Z^-,Z^+]}\ad_{Z^-}W^+\end{equation*}
Thus the right side of \eqref{eq:ads} is equal to $3 \ad_{Z^-}\ad_{[Z^-,Z^+]}\ad_{Z^-}W^+$, proving the proposition.
\end{proof}

\begin{cor}\label{cor:order2again}

If $p$ and $q$ have only  even-order terms, then as endomorphisms of $\m$, $p(\ad_{JZ})$ commutes with $q(\ad_Z)$.  \qed
\end{cor}

Throughout this paper, we will apply Lemma \ref{lemma:order2}(2) and   Corollary \ref{cor:order2again} to the maps $E_Z$ and $S_Z$ defined below.

\begin{lemma}\label{lemma:projectionlemma} 
Let $T:\m\rightarrow\m$ be $\mathbb C$-linear (which is to say, $\R$-linear and commuting with multiplication by $i$).  The following are equivalent:
\begin{enumerate} 
\item $T$ commutes with $J$.
\item $T$ commutes with $\frac 12(I-iJ)$ (projection from $\m$ to $\m^+$).   
\item $T$ commutes with $\frac 12(I+iJ)$ (projection from $\m$ to $\m^-$). 
\item $T(\m^+)\subseteq \m^+$ and $T(\m^-)\subseteq \m^-$.
\end{enumerate}
\end{lemma}

\begin{proof}  Each of the conditions 1, 2, 3 is equivalent to the statement that for arbitrary $W=W^++W^-\in \m$, the $\m^-$ part of $T(W^+)$ equals the $\m^+$ part of $T(W^-)$.  If this condition holds, then by taking $W^-=0$ (resp. $W^+=0$), we obtain the two statements in condition 4.  On the other hand, condition 4 implies that   the $\m^-$ part of $T(W^+)$ equals the $\m^+$ part of $T(W^-)$ since both are zero.  
\end{proof}

\begin{lemma}\label{lemma:projectioncommutes}
Let $Z\in\m$, and  let $p$ be a series with only even order terms. If  $p(\ad_Z)$ and $p(\ad_{JZ})$ converge as operators on $\m$, then   
$p(\ad_Z)\circ p(\ad_{JZ}):\m\rightarrow\m$ satisfies the conditions of Lemma \ref{lemma:projectionlemma}.  
\end{lemma}

\begin{proof}  Using  Lemma \ref{lemma:order2}(2) and Corollary \ref{cor:order2again} we see that the operator commutes with $J$.  
\end{proof}

\subsection{The maps $E_Z$ and $F_Z$}\label{ss:EZFZ}

Beginning with the exponential map $\exp:\g\rightarrow G$ and with fixed $Z\in\g$, we have the linear map 
$d\exp_Z:\g\simeq T_Z\g\rightarrow T_{\exp Z}G$.  There exist linear endomorphisms of $\g$, denoted 
$E_{Z,l}$ and $E_{Z,r}$, such that $d\exp_Z=dl_{\exp Z}\circ E_{Z,l}=dr_{\exp Z}\circ E_{Z,r}$.  We regard 
$\exp\left(tE_{Z,l}(A)\right)$ as giving a first-order approximation of $\exp(-Z)\exp(Z+tA)$, and 
$\exp\left(tE_{Z,r}(A)\right)$ as giving a first-order approximation of $\exp(Z+tA)\exp(-Z)$.  It is clear that 
$E_{Z,r}=\Ad_{\exp Z}\circ E_{Z,l}$.  It is also well known that the power series expansions are 
$E_{Z,l}=\sum_{n=0}^\infty\frac{\left(-\ad_Z\right)^n}{(n+1)!}$ and $E_{Z,r}=\sum_{n=0}^\infty\frac{\left(\ad_Z\right)^n}{(n+1)!}$.  Considering the even and odd parts, we define $E_Z=\sum_{n=0}^{\infty}\frac{\left(\ad_Z\right)^{2n}}{(2n+1)!}$ and $F_Z=\sum_{n=1}^{\infty}\frac{\left(\ad_Z\right)^{2n-1}}{(2n)!}$; then 
$E_{Z,l}=E_Z-F_Z$ and $E_{Z,r}=E_Z+F_Z$.  
For $Z\in\m$, we can regard $E_Z$ as an endomorphism of $\m$.  

We note the geometric significance of $E_Z$ for the noncompact Hermitian symmetric space (the argument for the compact case being analogous).  Let $\Exp:\m_0\rightarrow G_0/K_0$ denote the exponential map at the base point of $G_0/K_0$ (in the sense of Riemannian geometry); then $d\Exp_Z:\m_0\simeq T_Z\m_0\rightarrow T_{\Exp Z}\left(G_0/K_0\right)$ satisfies
$d\Exp_Z=dl_{\exp Z}\circ E_Z$.


\subsection{A decomposition of $\Ad_{\exp Z}\Upsilon$}\label{ss:Ad}

\begin{lemma}\label{lemma:decomposition} 
Let $Z\in\m$.  
\begin{enumerate}
\item The $\k$-part of $\Ad_{\exp Z}\Upsilon$ is 
$$\proj{\g}{\k}\left(\Ad_{\exp Z}\Upsilon\right)=\frac 12\left(\Ad_{\exp Z}+\Ad_{\exp \sigma\theta Z}\right)\Upsilon
=\frac 12\left(\Ad_{\exp Z}+\Ad_{\exp (-Z)}\right)\Upsilon
=\left(\Upsilon-F_Z(JZ)\right)
$$
\item The $\m$-part of of $\Ad_{\exp Z}\Upsilon$ is 
$$\proj{\g}{\m}\left(\Ad_{\exp Z}\Upsilon\right)=\frac 12\left(\Ad_{\exp Z}-\Ad_{\exp \sigma\theta Z}\right)\Upsilon
=\frac 12\left(\Ad_{\exp Z}-\Ad_{\exp (-Z)}\right)\Upsilon
=-E_Z(JZ)
$$
\end{enumerate}
Suppose further that $Z\in\m_0$, so that $\Ad_{\exp Z}\Upsilon\in\g_0=\k_0\oplus\m_0$.    
Without loss of generality, assume that 
$Z=\sum c_\psi x_\psi\in\a_0$.  
\begin{enumerate}
\item The $\k_0$-part of $\Ad_{\exp Z}\Upsilon$ also can be described as  
$$\begin{aligned}\proj{\g_0}{\k_0}\left(\Ad_{\exp Z}\Upsilon\right)&=
\frac 12\left(\Ad_{\exp Z}+\Ad_{\exp \theta Z}\right)\Upsilon\\
&=\Upsilon'+\frac i2\sum\cosh(2c_\psi)h_\psi=\Upsilon+\sum\left(\cosh(2c_{\psi})-1\right)\Upsilon_{\psi}\end{aligned}$$
\item The $\m_0$-part of of $\Ad_{\exp Z}\Upsilon$ also can be described as   
$$
\proj{\g_0}{\m_0}\left(\Ad_{\exp Z}\Upsilon\right)
=\frac 12\left(\Ad_{\exp Z}-\Ad_{\exp \theta Z}\right)\Upsilon
=-\frac 12\sum\sinh(2c_\psi)y_\psi$$
\end{enumerate}

\end{lemma}

\begin{proof}  
The essential calculation  is  $$\begin{aligned}\Ad_{\exp Z}\Upsilon&=\sum_{n=0}^\infty\frac{\ad_Z^{2n}}{(2n)!}\Upsilon+\sum_{n=0}^\infty\frac{\ad_Z^{2n+1}}{(2n+1)!}\Upsilon=\Upsilon+\sum_{n=1}^\infty\frac{\ad_Z^{2n-1}}{(2n)!}\ad_Z\Upsilon+\sum_{n=0}^\infty\frac{\ad_Z^{2n}}{(2n+1)!}\ad_Z\Upsilon\\&=\left(\Upsilon-F_Z(JZ)\right)-E_Z(JZ)\end{aligned}$$  
\end{proof} 


We shall need a variant of Lemma \ref{lemma:decomposition} for $2\times 2$ matrices, the proof of which is an easy computation:
\begin{lemma}\label{lemma:sl2}   
Let $X=\left(\begin{matrix}0&a\\b&0\end{matrix}\right)$ with $a,b\in\C^*$, and let $\sqrt{ab}$ denote either of the complex square roots of $ab$. 
\begin{enumerate}
\item $\exp(X)=\left(\begin{matrix} \cosh(\sqrt{ab})&\dfrac{a\sinh(\sqrt{ab})}{\sqrt{ab}}\\
\dfrac{b \sinh(\sqrt{ab})}{\sqrt{ab}}&\cosh(\sqrt{ab})\end{matrix}\right)$
\item $\Ad_{\exp X}\left(\begin{matrix}1&0\\0&-1\end{matrix}\right) = 
\left(\begin{matrix} \cosh(2\sqrt{ab})&-\dfrac{a\sinh(2\sqrt{ab})}{\sqrt{ab}}\\
\dfrac{b \sinh(2\sqrt{ab})}{\sqrt{ab}}&-\cosh(2\sqrt{ab})\end{matrix}\right)$ \qed
\end{enumerate}
\end{lemma}

\subsection{The map $S_Z$}\label{ss:SZ}
For $Z\in\m$, we define $S_Z=\frac 12\left(\Ad_{\exp Z}+\Ad_{\exp Z}^{-1}\right)$.  
We have already encountered this expression, having seen in Lemma \ref{lemma:decomposition}  that 
$\frac 12\left(\Ad_{\exp Z}+\Ad_{\exp Z}^{-1}\right)(\Upsilon)
=\left(\Upsilon-F_Z(JZ)\right)$.  However, in this section, we are interested in applying $S_Z$ to elements of $\m$.

Since $\Ad_{\exp Z}^{-1}=\Ad_{\exp(-Z)}$, $S_Z$  coincides with the even part of $\Ad_{\exp Z}$, so $S_Z=\sum_{n=0}^\infty\frac{\ad_Z^{2n}}{(2n)!}$.  If $Z\in\m_0$, we can  regard $S_Z$ as an $\R$-linear map from $\m_0$ to $\m_0$. 
We will see in Section \ref{s:eigen} that $\left.S_Z\right|_{\m_0}$ is invertible.  
\begin{lemma}\label{lemma:SZ}

Let $Z\in\m_0$, and consider $S_Z:\m_0\rightarrow\m_0$. 
 \begin{enumerate}
 \item $S_Z=\text{proj}^{\g_0}_{\m_0}\circ\Ad_{\exp Z}=\text{proj}^{\g_0}_{\m_0}\circ\Ad_{\exp (-Z)}$. 
\item $S_{JZ}^{-1}JS_{JZ}=S_ZJS_Z^{-1}$, or equivalently, 
$JS_{JZ}S_Z=S_{JZ}S_ZJ$.
\item $S_{JZ}\circ S_Z:\m\rightarrow\m$ commutes with the projections of $\m$ onto $\m^+$ and $\m^-$.  
\end{enumerate}
\end{lemma}

\begin{proof} For (1), give $W\in\m_0$, we have
$$\begin{aligned}\text{proj}^{\g_0}_{\m_0}\circ\Ad_{\exp Z}W&=\left(\frac{I-\theta}2\right)\left(\Ad_{\exp Z}W\right)=\frac 12\left(\Ad_{\exp Z}W-\Ad_{\exp(-Z)}(-W)\right)\\
&=
\frac 12\left(\Ad_{\exp Z}W+\Ad_{\exp(-Z)}W\right)
=S_ZW.
\end{aligned}$$
Part (2)  uses Lemma \ref{lemma:order2}(2) and Corollary \ref{cor:order2again}.     Part (3) is a special case of   Lemma 
\ref{lemma:projectioncommutes}(2).  
\end{proof}

The example of $SL(2)$ already shows that in general, $S_Z\neq S_{JZ}$.

\begin{lemma}\label{lemma:identzero}
Let $Z\in\m$. 
\begin{enumerate}
\item As a function with domain $\m$, $\ad_\Upsilon\circ\left(\Ad_{\exp Z}-\Ad_{\exp(-Z)}\right)$ is identically zero. 
\item As functions from $\m$ to $\m$, we have $S_ZJS_Z=\frac 12\ad_{\left(\left(\Ad_{\exp Z}+\Ad_{\exp(-Z)}\right)\Upsilon\right)}$. 
\end{enumerate}
\end{lemma}

\begin{proof}
If $A\in\m$, then $\sigma\theta\left(\left(\Ad_{\exp Z}-\Ad_{\exp(-Z)}\right)A\right)=\left(\Ad_{\exp(-Z)}-\Ad_{\exp Z}\right)(-A)$, showing that  $\left(\Ad_{\exp Z}-\Ad_{\exp(-Z)}\right)A\in\k$, whence (1). 

To prove (2), we take the difference of the two sides of the equation, apply to an arbitrary element $A\in\m$, and multiply by four, obtaining
\begin{eqnarray*}
&\left(\Ad_{\exp Z}+\Ad_{\exp(-Z)}\right)\circ\ad_\Upsilon\circ\left(\Ad_{\exp Z}+\Ad_{\exp(-Z)}\right)A
-2[\left(\Ad_{\exp Z}+\Ad_{\exp(-Z)}\right)\Upsilon,A]\\
&=\left(\Ad_{\exp Z}+\Ad_{\exp(-Z)}\right)\circ\ad_\Upsilon\circ\left(\Ad_{\exp Z}+\Ad_{\exp(-Z)}\right)A
-2\Ad_{\exp Z}\circ\ad_\Upsilon\circ\Ad_{\exp(-Z)}A\\&-2\Ad_{\exp(-Z)}\circ\ad_\Upsilon\circ\Ad_{\exp Z}A\\
&=\left(\Ad_{\exp Z}-\Ad_{\exp(-Z)}\right)\circ\ad_{\Upsilon}\circ\left(\Ad_{\exp Z}-\Ad_{\exp(-Z)}\right)A
\end{eqnarray*}
which is zero by (1). 
\end{proof}

We have seen that  $S_Z=\sum_{n=0}^\infty\frac{\ad_Z^{2n}}{(2n)!}=\ad_Z\circ F_Z + I$, and is  the even part of $\Ad_{\exp Z}$.  Analogously, we have:

\begin{lemma}\label{lemma:oddpart}
Let $Z\in \m$.  The odd part of $\Ad_{\exp Z}$ maps $\k$ to $\m$ and $\m$ to $\k$, and is given by 
$$\frac 12\left(\Ad_{\exp Z}-\Ad_{\exp(-Z)}\right)=\sum_{n=0}^\infty\frac{\ad_Z^{2n+1}}{(2n+1)!}=\ad_Z\circ E_Z\qed$$  
\end{lemma}

\subsection{The map $P$}\label{ss:P}

We define the holomorphic\footnote{with respect to the complex structure given by $i$} map $P:\m\rightarrow\m$ by  $Z\mapsto \ad_\Upsilon\Ad_{\exp Z}\Upsilon$.  

Since $K$ fixes $\Upsilon$ in the adjoint representation, $P$ is $K$-equivariant. 
Since $\Upsilon$ is fixed by $\sigma$ and $\theta$, $P$ commutes with $\sigma$ and $\theta$, so that 
$P$ restricts to (real analytic) maps from $\m_0$ to $\m_0$ and from $\m_u$ to $\m_u$.

Since $\ad_\Upsilon$ annihilates $\k$, we have that {\it $P(Z)$  coincides with 
$J$ applied to the $\m$-part of $Ad_{\exp Z}\Upsilon$. } Thus by Lemma \ref{lemma:decomposition}, 
$$P(Z)=J\proj{\g}{\m}\left(\Ad_{\exp Z}\Upsilon\right)=
\frac 12J\left(\Ad_{\exp Z}-\Ad_{\exp (-Z)}\right)\Upsilon
=-JE_Z(JZ)=E_{JZ}Z,$$
the last equation using  Lemma \ref{lemma:order2}. 
Additionally, if $Z=\sum c_\psi x_\psi\in\a_0$, by Lemma \ref{lemma:decomposition} we have 
$P(Z)
=\frac 12\sum\sinh(2c_\psi)x_\psi$.  
Since $\sinh:\R\rightarrow\R$ is bijective, the last expression shows that $\left. P\right|_{\a_0}:\a_0\rightarrow \a_0$ is a real analytic isomorphism, and $P$ maps $\a$ to $\a$ 
(a holomorphic map that is not bijective).

\begin{lemma}\label{lemma:PandJcommute}

For $P:\m\rightarrow\m$ as above, 
\begin{enumerate}
\item $P$ commutes with $J$. 
\item If $Z\in\m^+$ or $\m^-$, then $P(Z)=Z$. 
\item $P$ is an odd function. 
\end{enumerate}
\end{lemma} 

\begin{proof}  For (1), from Lemma \ref{lemma:order2}(2) we have 
  $E_{JZ}=-JE_ZJ$, and then  
$P(JZ)=E_{-Z}JZ=E_ZJZ=JE_{JZ}Z=JP(Z)$.  For (2), we note that if $Z\in\m^{\pm}$, then 
$JZ=\pm iZ$, and then $P(Z)=E_{JZ}Z=E_{\pm iZ}Z=\sum_{n=0}^\infty\frac{\left(-\ad_Z^2\right)^n}{(2n+1)!}Z$, and the terms vanish for $n\geq 1$. Part (3) is trivial.  
\end{proof}

\begin{prop}\label{prop:Pinvertible}
$P:\m_0\rightarrow\m_0$ is invertible.  
\end{prop}

\begin{proof}
We need three background facts: \begin{enumerate}
\item Each element of $\m_0$ is conjugate by $K_0$ to an element of $\a_0$. 
\item The intersection of the $K_0$-orbit of a point $A\in\a_0$ with $\a_0$ is a 
$N_{K_0}(\a_0)$-orbit. 
\item If two points in $\a_0$ have the same $N_{K_0}(\a_0)$-isotropy, then they have the 
same $K_0$-isotropy. 
\end{enumerate}

Item (1) already was stated in \S\ref{s:notation}.  We'll sketch a justification of (2) and (3).  By the Luna-Richardson Theorem, we have that the inclusion $\mathfrak a\subset\mathfrak m$ induces an isomorphism between $\quot{\mathfrak a}{N_K(\mathfrak a)}$ and $\quot{\mathfrak m}K$ (where the double slashes indicate spaces of closed orbits); moreover, the isomorphism carries  the $N_K(\mathfrak a)$-isotropy strata in $\quot{\mathfrak a}{N_K(\mathfrak a)}$ to the $K$-isotropy strata in $\quot{\mathfrak m}K$.  The isomorphism $\quot{\mathfrak a}{N_K(\mathfrak a)}\simeq\quot{\mathfrak m}K$ shows that 
the $K$-orbit of a point in $\a$ intersects $\a$ in an $N_K(\a)$-orbit;
the isotropy result shows that if two points in $\a$ have the same $N_K(\a)$-isotropy, then they have the same $K$-isotropy.  Note that (2) and (3) are versions of these statements over $\R$.  While the transition to actions of real groups can lead to unexpected phenomena, (2) and (3) follow due to special properties of {\em compact} real forms \cite{Br0}.  

We now proceed to the proof of Proposition \ref{prop:Pinvertible}.  
By (1) and the fact that $\left. P\right|_{\a_0}:\a_0\rightarrow\a_0$ is surjective, we have that $P$ is surjective. 

We prove that $P$ is injective.  Let $k\in K_0$ and $A,A'\in\a_0$.  We must show that if 
$P(A')=P(\Ad_k A)$, then $A'=\Ad_k A$. 

Since $P(A)\in\a_0$ and $\Ad_k P(A)=P(A')\in\a_0$, by (2), there exists $n\in N_{K_0}(\a_0)$ such that $\Ad_kP(A)=\Ad_nP(A)$.    Since $P(\Ad_nA)=P(A')$ and since $A,\Ad_nA\in\a_0$, we conclude that $\Ad_nA=A'$. 

Writing $m=k^{-1}n\in K_0$, our goal is to show that $\Ad_m A=A$, given that $\Ad_mP(A)=P(A)$.  By (3), it is enough to verify that $A$ and $P(A)$ have the same $N_{K_0}(\a_0)$-isotropy.  For this, we recall, that the Weyl group $N_{K_0}(\a_0)/Z_{K_0}(\a_0)$ acts on $\a_0=\span_{\R}\{x_\psi\}$ as signed permutations, and observe that for 
$A=\sum c_\psi x_\psi$ and $P(A)=\sum \frac{\sinh(2c_\psi)}2x_\psi$, 
we have $c_{\psi'}=c_\psi$ (resp. $-c_\psi$) if and only if 
$\frac{\sinh(2c_\psi')}2=\frac{\sinh(2c_\psi)}2$ (resp $-\frac{\sinh(2c_\psi)}2$), using that the hyperbolic sine function is bijective and odd.
\end{proof}

\begin{lemma}\label{lemma:differentialP}
The differential $dP_Z:\m_0\rightarrow\m_0$ is given by $S_{JZ}\circ E_Z$. 
\end{lemma}

\begin{proof}
Let $B\in\m_0$ and let $\sim$ denote the equivalence relation ``having the same tangent vector at $t=0$.''  We then have
$P(Z+tB)=\ad_\Upsilon\Ad_{\exp(Z+tB)}\Upsilon\sim\ad_\Upsilon\Ad_{\exp Z}\left(\left(I+t\,\ad_{E_{Z,l}B}\right)\Upsilon\right)$, so we are reduced to showing that 
$\ad_\Upsilon\circ \Ad_{\exp Z}\circ\ad_{E_{Z,l}B}\Upsilon=S_{JZ}E_ZB$. 
Now $E_{Z,l}B=E_ZB-F_ZB$, with $F_ZB\in\k_0$, hence 
$\ad_{E_{Z,l}B}\Upsilon=\ad_{E_ZB}\Upsilon=-J\circ E_Z(B)\in\m_0$. Finally, 
$\ad_\Upsilon\circ \Ad_{\exp Z}\circ\ad_{E_{Z,l}B}\Upsilon=
-\ad_\Upsilon\circ\Ad_{\exp Z}\circ J\circ E_Z(B)
=-J\circ\proj{\g_0}{\m_0}\Ad_{\exp Z}\circ J\circ E_Z(B)
=-J\circ S_Z\circ J\circ E_Z(B)=S_{JZ}\circ E_Z(B)$. 
\end{proof}

\subsection{The maps $t_\lambda$ and $T_\lambda$}\label{ss:t}

\begin{defn}\label{d:t}  
Let $\lambda\in\C^*$.  We define a $\C$-linear bijective endomorphism $t_\lambda$ of $\g$ by $t_\lambda$ equalling the identity map on $\k$, multiplication by $\lambda$ on $\m^+$, and multiplication by $\lambda^{-1}$ on $\m^-$. 
\end{defn}

\begin{lemma}\label{l:t}
\begin{enumerate}
\item For any $c\in\C$ with $e^{ci}=\lambda$, we have $t_\lambda=\Ad_{exp(c\Upsilon)}$. 
\item The map $t_\lambda$ is a Lie algebra isomorphism, and commutes with $\Ad_k$ for all $k\in K$. 
\item $t_\lambda(\m_0)=\{\lambda Z^++\lambda^{-1}\sigma Z^+:Z^+\in \m^+\}=\{Z^++\frac 1{|\lambda|^2}\sigma Z^+:Z^+\in\m^+\}$, and these equal the real form of $\m$ defined by the complex conjugation $Z^+\mapsto\frac 1{|\lambda|^2}\sigma Z^+$, $Z^-\mapsto |\lambda|^2\sigma Z^-$. 
\item If $Z\in\m$, then $\ad_{t_\lambda Z}^2=t_\lambda \circ\ad_Z^2 \circ t_{1/\lambda}$ as operators on $\m$. 
\item If $p$ is a power series with even order terms, then  $p(\ad_{t_\lambda Z})=t_\lambda\circ p(\ad_Z)\circ t_{1/\lambda}$ as operators on $\m$.  
\item With notation as in Definition \ref{d:stereo}, $I+i\lambda_3J=t_{\lambda}(\lambda_1-\lambda_2J)$ as operators on $\m$.  
\item If $A\in\m$, then $\theta(t_\lambda A)=t_{\overline\lambda}^{-1}\theta A$ and $\sigma(t_\lambda A)=t_{\overline\lambda}^{-1}\sigma A$.
\end{enumerate}
\end{lemma}

\begin{proof}
For Part 1, note that $\Ad_{\exp(c\Upsilon)}=\exp(\ad_{c\Upsilon})$, and recall that $c\Upsilon$ acts via $\ad$ on $\m^+$, $\k$, and $\m^-$ with eigenvalues $ci$, $0$, and $-ci$, respectively.  Part 2 follows immediately from Part 1.  Part 3 is a simple computation.  Part 4 is proved by computing that both operators act as $\lambda^2\ad_{Z^+}^2+\frac 1{\lambda^2}\ad_{Z^-}^2+\ad_{Z^+}\ad_{Z^-}+\ad_{Z^-}\ad_{Z^+}$ on $\m$, and then Part 5 follows immediately.  Part 6 is proved by computing that the operators both act as multiplication by $\frac{2|\lambda|^2}{1+|\lambda|^2}$ on $\m^+$ and as multiplication by $\frac 2{1+|\lambda|^2}$ on $\m^-$.  Part 7 is an easy computation.  
\end{proof}

\begin{defn}\label{d:T}  
For $\lambda\in\C^*$, we define $T_\lambda:\O\rightarrow\O$ by the rule $\Ad_{g_u}\Ad_{\exp Z}\Upsilon\mapsto \Ad_{g_u}\Ad_{\exp t_\lambda Z}\Upsilon$. 
\end{defn}

It is easy to see that $T_\lambda$ is well-defined and $G_u$-equivariant. 

\begin{prop}\label{p:TBijection}  
$T_\lambda$ is a bijection. 
\end{prop}

\begin{proof}
This will be restated below as Corollary \ref{cor:Tlambdabijection} and proved in Section \ref{s:proofs}, as a consequence of an explicit formula for $T_\lambda$ (Theorem \ref{thm:Tlambdaformula}).
\end{proof}

\subsection{The Killing form}\label{ss:killingform}

\begin{lemma}\label{lemma:killingform}

\begin{enumerate}
\item For all $A,B\in\g$, we have $\kappa([\Upsilon, A],B)=\kappa(\Upsilon,[A,B])=-\kappa(A,[\Upsilon,B])$. 
\item For all $A,B\in\m$, we have $\kappa(JA,B)=-\kappa(A,JB)$ and $\kappa(JA,JB)=\kappa(A,B)$. 
\item If $Z\in\m$, and if $p$ is a power series with only even order terms, such that $p(\ad_Z)$ converges as an operator on $\m$ (e.g., $p(\ad_Z)=E_Z$ or $S_Z$), then $\kappa\left(p(\ad_Z)A,B\right)=\kappa\left(A,p(\ad_Z)B\right)$. 
\item If $Z\in\m_0$, and $A,B\in\m$ are eigenvectors of $\ad_Z^2$ with eigenvalues $\nu_1\neq\nu_B$, then 
$\kappa(A,B)=0$. 
\item If $\lambda\in\C^*$ and $C,D\in\m$, then $\kappa(t_\lambda C, t_\lambda D)=\kappa (C,D)$.  
\item If $\lambda\in\C^*$ and $C,D\in\m_0$, then $\kappa(C,D)=2\Re(\kappa(C^+,D^-))$ and $\kappa(t_\lambda C, D)=\lambda\kappa(C^+,D^-)+\frac 1\lambda\overline{\kappa(C^+,D^-)}$.  
\end{enumerate}
\end{lemma}

\begin{proof}
(1) is the associativity of the Killing form, and (2) follows since $J=\ad_\Upsilon$ on $\m$ and $J^2=-I$. 
The proof of (3) reduces to the case $\kappa\left(ad_Z^2A,B\right)=\kappa\left(A,\ad_Z^2B\right)$, which follows from $\kappa\left(\ad_ZA,B\right)=-\kappa\left(A,\ad_ZB\right)$. Part (4) follows easily from (3).  Part (5) follows immediately from Lemma \ref{l:t}(1), or by decomposing $t_\lambda C$ and $t_\lambda D$ into their $\m^+$ and $\m^-$ parts and using that $\m^+$ is orthogonal to $\m^-$ under the Killing form.  Part (6) is proved using the fact that $C=C^++\sigma(C^+)$, $D=D^-+\sigma(D^-)$, and the orthogonality of $\m^+$ and $\m^-$.  
\end{proof}


\section{Eigenvalues and eigenvectors}\label{s:eigen}   
(Beyond Theorem \ref{theorem:eigenvaluetheorem}, this section is not needed in the sequel.)

The data for the hyperk\"ahler structure (\S \ref{s:quaternionic}, \S \ref{s:hyper}) rely on the maps $E_Z$,  $E_{JZ}$, $S_Z$, and $S_{JZ}$,  for $Z\in\m_0$.  Since these are power series in $\ad_Z^2$ or $\ad_{JZ}^2$, one can understand their actions on $\m_0$ once one understands the eigenvalues and eigenvectors of $\ad_Z^2$ and $\ad_{JZ}^2$ on $\m_0$.  The computations in this section address that question.

From Proposition  \ref{prop:squarescommute}, we conclude that 
$\ad_Z^2$ and $\ad_{JZ}^2$ are {\em simultaneously} diagonalizable on $\m_0$ (with nonnegative real eigenvalues), and then  $E_Z$,  $E_{JZ}$, $S_Z$, and $S_{JZ}$,  are simultaneously diagonalizable.

From 
$\ad_{JZ}^2=-J\ad_Z^2J$, it is easy to see that:

\begin{lemma}\label{lemma:eigenvaluepairs}
If $Z, A\in \m_0$, and if $\ad_Z^2A=\nu_1A$ and $\ad_{JZ}^2A=\nu_2A$, then 
$\ad_{Z}^2(JA)=\nu_2JA$ and $\ad_{JZ}^2(JA)=\nu_1JA$. \qed
\end{lemma}

With the notation of Lemma \ref{lemma:eigenvaluepairs}, using the power series for $E_Z$, $E_{JZ}$, $S_Z$, and $S_{JZ}$, we have 

\begin{theorem}\label{theorem:eigenvaluetheorem}
$A$ is an eigenvector   with eigenvalues $\frac{\sinh\sqrt{\nu_1}}{\sqrt{\nu_1}}$, $\frac{\sinh\sqrt{\nu_2}}{\sqrt{\nu_2}}$, $\cosh\sqrt{\nu_1}$, and $\cosh\sqrt{\nu_2}$ respectively, and $JA$ is an eigenvector with eigenvalues   
$\frac{\sinh\sqrt{\nu_2}}{\sqrt{\nu_2}}$, $\frac{\sinh\sqrt{\nu_1}}{\sqrt{\nu_1}}$, $\cosh\sqrt{\nu_2}$, and $\cosh\sqrt{\nu_1}$ respectively.\footnote{In particular, note that if $\nu_1=0$, then $E_Z(A)=A=S_Z(A)$.}  \qed
\end{theorem}

We see that $\m_0$ decomposes into $J$-stable subspaces on which the actions of $E_Z$, $E_{JZ}$, $S_Z$, and $S_{JZ}$ are easily understood, once we 
obtain the eigenvalues and eigenvectors of $\ad_Z^2$ and $\ad_{JZ}^2$ on $\m_0$.  That is the remaining goal of this section, culminating in Theorem \ref{thm:EigenvalueTheorem} and Corollary \ref{cor:EigenvalueCorollary}.  We  will see that these are expressed easily in terms of restricted roots.  More precisely, we  will describe the $J$-complex subspace $(\m_0)_{\{\nu_1,\nu_2\}}$ of $\m_0$ which is the direct sum of two (real) subspaces, namely, the first being the vectors on which  $\ad_Z^2=\nu_1$ and $\ad_{JZ}^2=\nu_2$, the second being the vectors on which $\ad_Z^2=\nu_2$ and $\ad_{JZ}^2=\nu_1$.  (Each of these subspaces is a real form of $(\m_0)_{\{\nu_1,\nu_2\}}$ relative to $\C=\R\oplus\R J$, and each is $J$ times the other.)  We also will identify the eigenvalue pairs $\{\nu_1,\nu_2\}$ that occur.  

Using the notation from \S\ref{s:notation}, we may assume that $Z=\sum\lambda_{\psi} x_\psi$ and $JZ=\sum \lambda_\psi y_\psi$, where the sums are over the maximal collection of strongly orthogonal positive $\m$-roots $\{\psi\}$.  We write $Z=(I+\sigma)Z^+$ and $A=(I+\sigma)A^+$.  By expanding $\ad_Z^2$ and $\ad_{JZ}^2$,  it is easily seen that the system of equations $\{\ad_Z^2 A=\nu_1 A, \ad_{JZ}^2 A=\nu_2 A\}$ is equivalent to the system
$$\ad_{Z^+}\ad_{\sigma Z^+}A^+=\frac{\nu_1+\nu_2}2 A^+\qquad\qquad(\ad_{Z^+}^2\circ\sigma) A^+=\frac{\nu_1-\nu_2}2 A^+$$
We shall consider the slightly simpler system
$$\ad_{Z^+}\ad_{\sigma Z^+}A^+=\frac{\nu_1+\nu_2}2 A^+  \qquad\qquad\ad_{Z^+}^2 \ad_{\sigma Z^+}^2 A^+=\left(\frac{\nu_1-\nu_2}2\right)^2 A^+$$
where the second equation is obtained by applying $\ad_{Z^+}^2\circ\sigma$ twice to $A^+$; we see that $A^+$ is a solution to this system exactly when $A\in(\m_0)_{\{\nu_1,\nu_2\}}$, and will write $A^+\in\m^+_{\{\nu_1,\nu_2\}}$.

Write $A^+=\sum A_\alpha$, where $\alpha$ runs through the positive restricted $\m$-roots.  
     
Examining the first equation in the system, we have that $$\ad_{Z^+}\ad_{\sigma Z^+} A^+=\ad_{[Z^+,\sigma Z^+]}A^+=\ad_{\sum \lambda_\psi^2 h_\psi}A^+=\sum_{\alpha}\ad_{\sum \lambda_\psi^2 h_\psi}A_\alpha=\sum_\alpha\alpha\left(\sum \lambda_\psi^2 h_\psi\right)A_\alpha$$
So the first equation in the system holds exactly when $\alpha\left(\sum \lambda_\psi^2 h_\psi\right)=\frac{\nu_1+\nu_2}2$ whenever $A_\alpha\neq 0$.  

Turning to the second equation in the system, we compute that 
$$\begin{aligned} \ad_{Z^+}^2\ad_{\sigma Z^+}^2 A^+&=\ad_{Z^+}\left(\ad_{[Z^+,\sigma Z^+]}+\ad_{\sigma Z^+}\ad_{Z^+}\right)\ad_{\sigma Z^+}A^+\\&=\ad_{Z^+}\ad_{[Z^+,\sigma Z^+]}\ad_{\sigma Z^+}A^++\left(\ad_{Z^+}\ad_{\sigma Z^+}\right)^2A^+\\
&=\ad_{Z^+}\ad_{[Z^+,\sigma Z^+]}\ad_{\sigma Z^+}A^++\left(\frac{\nu_1+\nu_2}2\right)^2 A^+\\
\end{aligned}$$
where the last step follows from the first equation in the system.  Thus we may replace the second equation in the system with 
$$\ad_{Z^+}\ad_{[Z^+,\sigma Z^+]}\ad_{\sigma Z^+}A^+=\left(\left(\frac{\nu_1-\nu_2}2\right)^2-\left(\frac{\nu_1+\nu_2}2\right)^2\right)A^+=-\nu_1\nu_2 A^+$$
We can express the left side of this equation in terms of roots, as follows.  (Sums with indices $\psi$ and $\gamma$ are over the maximal set of strongly orthogonal positive $\m$-roots.)

$$\begin{aligned}\ad_{Z^+}\ad_{[Z^+,\sigma Z^+]}\ad_{\sigma Z^+}A^+
&=\left(\ad_{\sum_\psi \lambda_\psi e_\psi}\right)\circ\left(\ad_{\sum_\gamma \lambda_\gamma^2 h_\gamma}\right)\circ\left(\ad_{\sum_\psi \lambda_\psi f_\psi}\right)\left(\sum_\alpha A_\alpha\right)\\
&=\left(\ad_{\sum_\psi \lambda_\psi e_\psi}\right)\circ\left(\ad_{\sum_\gamma \lambda_\gamma^2 h_\gamma}\right)\circ\left(\sum_\alpha\left(\sum_\psi \lambda_\psi[f_\psi,A_\alpha]\right)\right)\\
&=\left(\ad_{\sum_\psi \lambda_\psi e_\psi}\right)\circ\left(\sum_\alpha\sum_\psi \lambda_\psi \cdot (\alpha-\psi)\left(\sum_\gamma \lambda_\gamma^2 h_\gamma\right)\cdot [f_\psi,A_\alpha]\right)\\
&=\left(\ad_{\sum_\psi \lambda_\psi e_\psi}\right)\circ\left(\sum_\alpha \sum_\psi \left(\frac{\nu_1+\nu_2}2 -2\lambda_\psi^2\right)\cdot \lambda_\psi\cdot [f_\psi,A_\alpha]\right)\\
&=\sum_\alpha\sum_\psi\left(\frac{\nu_1+\nu_2}2-2\lambda_\psi^2\right)\cdot\lambda_\psi^2\cdot [h_\psi,A_\alpha]\\&\phantom{.}  \qquad \text{(since $[e_\gamma,[f_\psi,A_\alpha]]=[[e_\gamma,f_\psi],A_\alpha]=\delta_{\gamma,\psi}[h_\psi,A_\alpha]$)}\\
&=\sum_\alpha\left(\left(\frac{\nu_1+\nu_2}2\right)\cdot \alpha\left(\sum_\psi \lambda_\psi^2 h_\psi\right)-2\alpha\left(\sum_\psi \lambda_\psi^4 h_\psi\right)\right)A_\alpha\\
&=\sum_\alpha\left(\left(\frac{\nu_1+\nu_2}2\right)^2-2\alpha\left(\sum_\psi \lambda_\psi^4 h_\psi \right)\right)A_\alpha
\end{aligned}$$
This expression equals $-\nu_1\nu_2 A^+$ if and only if, for each $\alpha$ with $A_\alpha\neq 0$, we have $\alpha(\sum_\psi \lambda_\psi^4 h_\psi)=\frac{\nu_1^2+6\nu_1\nu_2+\nu_2^2}8$.

We summarize this argument with the following theorem. Using the notation developed through this section, we have:

\begin{theorem}\label{thm:EigenvalueTheorem}
Let $Z=\sum \lambda_{\psi}x_{\psi}$.  The space $\m^+_{\{\nu_1,\nu_2\}}$ is the direct sum of the restricted root spaces $\m^+_{\alpha}$, where $\alpha$ satisfies the two conditions $\alpha\left(\sum_\psi \lambda_\psi ^2 h_\psi\right)=\frac{\nu_1+\nu_2}2$ and $\alpha\left(\sum_\psi \lambda_\psi ^4 h_\psi\right)=\frac{\nu_1^2+6\nu_1\nu_2+\nu_2^2}8$.  \qed
\end{theorem}

Of course, for fixed $Z$, $(\m_0)_{\{\nu_1,\nu_2\}}=\{0\}$ for almost all choices of $\{\nu_1,\nu_2\}$, with the equations in Theorem \ref{thm:EigenvalueTheorem} predicting the exact choices of $\{\nu_1,\nu_2\}$ for which $(\m_0)_{\{\nu_1,\nu_2\}}$ is nonzero. By simple algebra, we have:

\begin{cor}\label{cor:EigenvalueCorollary}
Fix $Z$  as above, let $\alpha$ be a positive restricted $\m$-root, and write $x=\alpha\left(\sum_\psi \lambda_\psi ^2 h_\psi\right)$ and $y=\alpha\left(\sum_\psi \lambda_\psi ^4 h_\psi\right)$.    Then $\m^+_\alpha$ is contained in 
$\m^+_{\{\nu_1,\nu_2\}}$, where $\{\nu_1,\nu_2\}=\{x\pm\sqrt{2x^2-2y}\}$.   \qed
\end{cor}

Hence, if we desire to know the simultaneous eigenvalue pairs of $\ad_Z^2$ and $\ad_{JZ}^2$ on $\m_0$, we run through each positive restricted $\m$-root $\alpha$, compute the values of $x$ and $y$, and obtain the eigenvalue pairs $\{x\pm\sqrt{2x^2-2y}\}$.   We recall that from classification, the possibilities for $\alpha$ are $\psi_s$, $\frac 12(\psi_s+\psi_t)$ with $s<t$, and (for some but not all Hermitian symmetric spaces) $\frac 12\psi_s$.  Relative to $Z=\sum\lambda_\psi x_\psi$, we find $\{\nu_1,\nu_2\}=\{4\lambda_s^2,0\}$, $\{(\lambda_s\pm\lambda_t)^2\}$, and $\{\lambda_s^2\}$ (with $\nu_1=\nu_2$) respectively, for the three cases for $\alpha$.  


\section{Some geometric identifications}\label{s:geometry}
The purpose of this section is to present several realizations of the cotangent bundle to $G_u/K_0$, as well as identifications among them.

 \subsection{Cotangent bundles and twisted products}
 \label{ss:cotangenttwisted} 
 
 \subsubsection{Two twisted products}\label{sss:twotwisted}
 
 {\em We identify the real cotangent bundle to $G_u/K_0$ with $\twist{G_u}{K_0}{\m_u}$.}  As a cotangent vector, $[g_u,W]$ sends the tangent vector $\left.\frac{d}{dt}\right|_0g_u\exp(tA)K_0$ to $\kappa(A,W)$. (Here it is enough to consider $A\in\m_u$.)\footnote{Recall that if $A\subset B$ are groups and $A$ acts on a set $X$, then $\twist BAX$ is the space of $A$-orbits on $B\times X$, where $\phantom{.}^a(b,x)=(ba^{-1}, \phantom{.}^ax)$.  The image of $(b,x)$ in $\twist BAX$ is denoted $[b,x]$.  The twisted product $\twist BAX$ is a bundle over $B/A$ with fiber $X$.  }

{\em We identify the complex cotangent bundle  to $G/Q$ with $\twist{G}{Q}{\m^-}$.}   As a complex cotangent vector, $[g,W^-]$ sends the tangent vector $\left.\frac{d}{dt}\right|_0g\exp(tA^+)Q$ to $2\kappa(A^+,W^-)$.  Here $\twist{G}{Q}{\m^-}$ carries an obvious complex structure, 
which we denote by $J_3$.  Moreover, the obvious action of $G$ on $\twist{G}{Q}{\m^-}$ is holomorphic.  

\subsubsection{Identifying the two twisted products}\label{sss:identifytwotwisted}

As we now describe, {\em there is a simple $G_u$-equivariant bijection (denoted $\Theta$) between $\twist{G}{Q}{\m^-}$ and $\twist{G_u}{K_0}{\m_u}$.}    Note that by the  Iwasawa decomposition of $G$ (or the transitivity of $G_u$ on $G/Q$),  any $g\in G$ can be written as 
$g=g_uq$ for some   $g_u\in G_u$ and $q\in Q$.  We then have:

\begin{eqnarray*}
\Theta: \twist {G}{Q}{\m^-} & \rightarrow& \twist{G_u}{K_0}{\m_u}\\
   \left[g,W^-\right]=\left[g_uq,W^-\right]
   =\left[g_u,Ad_q(W^-)\right]&\mapsto &\left[g_u,(I+\theta)(Ad_q(W^-))\right]\\       
   \left[g_u,W^-\right]& \leftarrow &\left[g_u,(I+\theta)W^-\right]
   \end{eqnarray*}
(This tells us how $G$ acts on 
$ \twist{G_u}{K_0}{\m_u}$  $J_3$-holomorphically.)

We make an observation regarding the   compatibility of these identifications. 
Let $g_u\in G_u$, let $W=(I+\theta)W^-\in\m$ and let $A=(I+\theta)A^+\in\m$.  In $G_u/K_0$, the tangent vector given by the equivalence class of the curve  $g_u\exp(tA)K_0$ is sent by the cotangent vector
$[g_u,W]$ to $\kappa(A,W)=\kappa(A^+,W^-)+\kappa(\theta A^+,\theta W^-)=2\Re\left(\kappa(A^+,W^-)\right)$.  However, the curve $g_u\exp(tA)K_0\sim g_u\exp(tA^+)\exp(t\theta A^+)K_0$ is identified  in $G/Q$ with the curve $g_u\exp(tA^+)Q$, which  is sent by the cotangent vector $[g_u,W^-]$ to $2\kappa(A^+,W^-)$.

\subsection{Cotangent bundles and adjoint orbits}\label{ss:cotangentadjoint}

\subsubsection{$\O$ is the cotangent bundle to $\Ou$}\label{sss:OandOu}

Crucially, {\em the cotangent bundle to the Hermitian symmetric space  $\Ou$ can be identified with  $\O\subset\g$, the $\Ad(G)$-orbit of $\Upsilon$.}  Verifying that $\O$ gives the cotangent bundle to $\Ou$ relies on an orbit decomposition, as follows. By a decomposition result of Kobayashi \cite{Ko}, $G=G_u\times exp(\m_0)\times exp(i\k_0)$; from which it follows that $\O\simeq G/K
\simeq\twist{G_u}{K_0}{\m_0}$, which  is the cotangent bundle to $G_u/K_0\simeq\Ou$.  

We will constantly make use of the following fact (see \cite{Ko}): {\em every element of $\O$ is expressible as $\Ad_{g_u} \Ad_{\exp Z}\Upsilon$, where $g_u\in G_u$  and $Z\in\m_0$.} The projection of the cotangent bundle on the zero fiber is then given by $\O\rightarrow\Ou$, $\Ad_{g_u} \Ad_{\exp Z}\Upsilon\mapsto \Ad_{g_u}\Upsilon$.  

Note that $\O$ carries an obvious complex structure (which we will denote $J_1$) and relative to this complex structure, the obvious $G$-action is holomorphic (even algebraic).

\subsubsection{Two identifications of $\O$ and 
$\twist{G_u}{K_0}{\m_u}$}\label{sss:twoidentifications}

As indicated in Section \ref{sss:OandOu}, we have a bijection
\begin{eqnarray*}
\Phi_{\text{triv}}:\O&\rightarrow&\twist{G_u}{K_0}{\m_u}\\
\Ad_{g_u}\Ad_{\exp Z}\Upsilon&\mapsto&[g_u,Z]
\end{eqnarray*}

Unfortunately, $\Phi_{\text{triv}}$ is not the best isomorphism to use from the point of view of pseudok\"ahler structures, as indicated by the work of other authors.    
We also  will need a map $\Phi$, defined  by 
\begin{eqnarray*}
\Phi:\O&\rightarrow&\twist{G_u}{K_0}{\m_u}\\
\Ad_{g_u}\Ad_{\exp Z}\Upsilon&\mapsto&[g_u,J\Im\left(\Ad_{\exp Z}\Upsilon\right)].
\end{eqnarray*} 
We spot the appearance of the map $P$:  

\begin{lemma}\label{lemma:describePhi}

\begin{enumerate}
\item We have the equalities
\begin{eqnarray*}\Phi\left(\Ad_{g_u}\Ad_{\exp Z}\Upsilon\right)
&=&[g_u,J\Im\left(\Ad_{\exp Z}\Upsilon\right)]
=[g_u,-\frac{i}2J\left(\Ad_{\exp(Z)}-\Ad_{\exp(-Z)}\right)\Upsilon]\\
&=&[g_u,-i\,\ad_\Upsilon\Ad_{\exp Z}\Upsilon]
=[g_u,-iJ\text{proj}^\g_\m\left(\Ad_{\exp Z}\Upsilon\right)]\\
&=&[g_u,-iJ\text{proj}^{\g_0}_{\m_0}\left(\Ad_{\exp Z}\Upsilon\right)]
=[g_u,-i\, P(Z)]
=[g_u,-i\, E_{JZ}Z].
\end{eqnarray*}
\item If $Z=\sum c_\psi x_\psi\in\a_0$, then  $\Phi\left(\Ad_{g_u}\Ad_{\exp Z}\Upsilon\right)=[g_u,\frac 1{2i}\sum\sinh(2c_\psi)x_\psi]$. 
\item $\Phi$ is a bijection.
\item $\Phi^{-1}$ takes $[g_u,W]$ to $\Ad_{g_u}\Ad_{\exp\left(P^{-1}(-iW)\right)}\Upsilon$.
\end{enumerate}
\end{lemma}

\begin{proof}
For (1) and (2), see the equalities in the beginning of Section \ref{ss:P}.    Part (3) follows from the fact that $P:\m_0\rightarrow\m_0$ is a bijection (Proposition \ref{prop:Pinvertible}), and (4) follows from (1). 
\end{proof}

\section{Quaternionic structure}\label{s:quaternionic}
In this section, we give the complex structures $J_1$, $J_2$, $J_3$ on $\O\simeq\twist GQ{\m^-}$. We begin with notation for some easily-described tangent vectors.  (At various points in the development of the hyperk\"ahler structure, some of these are easier to work with than others.)  The complex structures themselves are defined in Definition \ref{defn:complexstructures}.   We describe the almost complex structures in terms of our various tangent vectors.

\subsection{Identifications of tangent spaces}\label{ss:tangents}
Fix $g_u\in G_u$ and $Z\in \m_0$.  Let $(I+\theta)W^-=W=-iP(Z)\in\m_u$.  Let $(A,B)\in\m_u\times\m_u$.  Write $A=A^++A^-$ with $A^+=\theta A^-$ and similarly for $B$.  We use these data to produce a variety of tangent vectors to $T^*(G_u/K_0)$, in its various manifestations. 

 Let ${(A,B)}\tilde{}=\left.\frac d{dt}\right|_0 \Ad_{g_u}\Ad_{\exp Z}\Ad_{\exp t(A+iB)}\Upsilon$, a tangent vector to $\O$ at $\Ad_{g_u}\Ad_{\exp Z}\Upsilon$. 

Let $(A,B)\hat{}=\left.\frac d{dt}\right|_0 \Ad_{g_u}\Ad_{\exp tA}\Ad_{\exp (Z+tiB)}\Upsilon$, a tangent vector to $\O$ at $\Ad_{g_u}\Ad_{\exp Z}\Upsilon$.

  Let $(A,B)^\sharp{}=\left.\frac d{dt}\right|_0[g_u\exp(tA),W+tB]$, a tangent vector to $\twist{G_u}{K_0}{\m_u}$ at $[g_u,W]$.

 Let $(A,B)^\flat{}=\left.\frac d{dt}\right|_0[g_u\exp(tA^+), W^-+tB^-]$, a tangent vector to $\twist GQ{\m^-}$ at $[g_u,W^-]$.  

\begin{lemma}\label{lemma:relatevectorfields}
These tangent vectors are related by the rules
\begin{enumerate}
\item
$(A,B)\hat{}=(S_ZA,E_ZB)\tilde{}$  
\item $d\Phi(A,B)\hat{}=(A,S_{JZ}E_ZB)^\sharp$
\item $d\Theta(A,B)^\flat{}=(A,B)^\sharp{}$.
\end{enumerate}
\end{lemma}

\begin{proof}
For (1), we compute that 
\begin{eqnarray*} (A,B)\hat{}\leftrightarrow \Ad_{g_u}\Ad_{\exp (tA)}\Ad_{\exp (Z+tiB)}\Upsilon&\sim&\Ad_{g_u}\Ad_{\exp (tA)}\Ad_{\exp (Z)}\Ad_{\exp (t E_{Z,l}(iB))}\Upsilon\\
&=&\Ad_{g_u}\Ad_{\exp (Z)}\Ad_{\exp (t\Ad_{\exp (-Z)}A)}\Ad_{\exp (tE_{Z,l}(iB))}\Upsilon
\end{eqnarray*}
  To prove that this gives the tangent vector $(S_ZA,E_ZB)\tilde{}$, we need to show that  $$\Ad_{\exp(t(E_{Z,l}(iB)+\Ad_{\exp(-Z)}A))}\Upsilon\sim \Ad_{\exp{t(S_ZA+iE_ZB)}}\Upsilon$$ but this is immediate from $[E_{Z,l}B,\Upsilon]=[E_ZB,\Upsilon]$ (Section \ref{ss:EZFZ}) and  $[\Ad_{\exp(-Z)}A,\Upsilon]=[S_ZA,\Upsilon]$ (Lemma 
\ref{lemma:SZ}(1)).

For (2), applying $\Phi$ to $\Ad_{g_u}\Ad_{\exp(tA)}\Ad_{\exp(Z+tiB)}\Upsilon$, we obtain $[g_u\exp(tA),-iP(Z+tiB)]$; then use Lemma \ref{lemma:differentialP}.

For (3), we compute that 
\begin{eqnarray*}\Theta^{-1}[g_u\exp(tA),W+tB]&=&[g_u\exp(tA),W^-+tB^-]
\sim[g_u\exp(tA^+)\exp(tA^-),W^-+tB^-]\\
&=&[g_u\exp(tA^+),\Ad_{\exp(tA^-)}(W^-+tB^-)]\\&=&[g_u\exp(tA^+),W^-+tB^-],
\end{eqnarray*}
 where at the last step, we use $[\m^-,\m^-]=0$.
\end{proof}

\subsection{Complex structures}\label{ss:complexstructures}

\begin{defn}\label{defn:Gamma}
We define $\Gamma:\O\rightarrow\O$ to be the bijection 
$\Ad_{g_u}\Ad_Z\Upsilon\mapsto \Ad_{g_u}\Ad_{JZ}\Upsilon$.
\end{defn}

\begin{defn}\label{defn:complexstructures}
We define the following complex structures on $\O$:
\begin{enumerate}
\item $J_1$ is the usual complex structure on $\O\simeq G/K$ (coming from the complex structure on $G$). 
\item $J_2$ is the complex structure on $\O$ defined so that the bijection 
$\Gamma:(\O,J_2)\rightarrow (\O,J_1)$  is holomorphic.  (See Definition \ref{defn:Gamma}.)
\item $J_3$ is the complex structure on $\O$ defined so that the bijection
$\Theta^{-1}\circ\Phi:(\O, J_3)\rightarrow\twist GQ{\m^-}$ is holomorphic, where $\twist GQ{\m^-}$
has the usual complex structure coming from $G$, $Q$, and $\m^-$.  (Here the almost complex structure on $\m^-$ is $i$, not $-i$.)
\end{enumerate}
By abuse of notation, we define complex structures $J_1$, $J_2$, and $J_3$ on $\twist GQ{\m^-}$ so that 
$\Theta^{-1}\circ\Phi$ is holomorphic for the corresponding complex structure on $\O$, and on $\twist{G_u}{K_0}{\m_u}$ so that $\Phi$ is holomorphic.  
\end{defn}

\begin{lemma}\label{lemma:J3}

\phantom{.}

\begin{enumerate}
\item $J_1(A,B)\tilde{}=(-B,A)\tilde{}$
\item $J_3(A,B)^\flat{}=(JA,-JB)^\flat{}$
\item $J_3(A,B)^\sharp=(JA,-JB)^\sharp$
\end{enumerate}
\end{lemma}
\begin{proof}
The first two statements are immediate, and the third follows from Lemma \ref{lemma:relatevectorfields}(3).
\end{proof}

We then  compute the almost complex structures for the `hat' tangent vectors:

\begin{lemma}\label{lemma:hatalmostcomplexstructures}
\phantom{.}
\begin{enumerate}
\item $J_1(A,B)\hat{}=(-S_Z^{-1}E_ZB,E_Z^{-1}S_ZA)\hat{}$
\item $J_2(A,B)\hat{}=(-JS_Z^{-1}E_ZB,-E_Z^{-1}S_ZJA)\hat{}$
\item $J_3(A,B)\hat{}=(JA,-E_Z^{-1}S_{Z}JS_{Z}^{-1}E_ZB)\hat{}$
\end{enumerate}
\end{lemma}

\begin{proof}
(1) We have that 
$$J_1(A,B)\hat{}=J_1(S_ZA,E_ZB)\tilde{}=(-E_ZB,S_ZA)\tilde{}=(-S_Z^{-1}E_ZB,E_Z^{-1}S_ZA)\hat{}.$$  

(2) 
Using (1), it is easy to compute that 
$d\Gamma^{-1}\circ J_1\circ d\Gamma$ carries  
$$\left.\frac{d}{dt}\right|_0Ad_{g_u}\Ad_{\exp(tA)}
\Ad_{\exp(Z+tiB)}\Upsilon\qquad\text{ to }\qquad\left.\frac{d}{dt}\right|_0\Ad_{g_u}\Ad_{\exp(tA')}\Ad_{\exp (Z+tiB')}\Upsilon,$$ where
$$(A',B')=(-S_{JZ}^{-1}E_{JZ}JB,-JE_{JZ}^{-1}S_{JZ}A)=(-JS_Z^{-1}E_ZB,-E_Z^{-1}S_ZJA)$$

(3)  We use Lemmas \ref{lemma:J3} and \ref{lemma:relatevectorfields}, and compute
$$(A,B)\hat{}\quad \underset{d\Phi}{\mapsto}\quad (A,S_{JZ}E_ZB)^\sharp\quad\underset{J_3}{\mapsto}\quad(JA,-JS_{JZ}E_ZB)^\sharp\quad\underset{d\Phi^{-1}}{\mapsto}\quad(JA,-E_Z^{-1}S_{JZ}^{-1}JS_{JZ}E_ZB)\hat{}$$
which equals $(JA,-E_Z^{-1}S_{Z}JS_{Z}^{-1}E_ZB)\hat{}$ by Lemma \ref{lemma:SZ}(2).
\end{proof}

Lemma \ref{lemma:hatalmostcomplexstructures} immediately yields 
that $(J_1,J_2,J_3)$ comprise a quaternionic structure:

\begin{lemma}\label{lemma:quaternionic}
$J_1J_2=J_3=-J_2J_1$. \qed
\end{lemma}

Lemma \ref{lemma:relatevectorfields} allows us to express the almost complex structures in Lemma \ref{lemma:hatalmostcomplexstructures} for the `tilde' and `sharp' tangent vectors.  
The computations are trivial. 
\begin{lemma}\label{lemma:tildesharpJ}
For the `tilde' tangent vectors:
\begin{enumerate}
\item $J_1(A,B)\tilde{}=(-B,A)\tilde{}$
\item $J_2(A,B)\tilde{}=(-S_ZJS_Z^{-1}B,-S_ZJS_Z^{-1}A)\tilde{}$
\item $J_3(A,B)\tilde{}=(S_ZJS_Z^{-1}A,-S_ZJS_Z^{-1}B)\tilde{}$
\end{enumerate}
For the `sharp' tangent vectors:
\begin{enumerate}
\item $J_1(A,B)^\sharp=(-S_Z^{-1}S_{JZ}^{-1}B,S_{JZ}S_ZA)^\sharp   $
\item $J_2(A,B)^\sharp=(-JS_Z^{-1}S_{JZ}^{-1}B,-S_{JZ}S_ZJA)^\sharp   $
\item $J_3(A,B)^\sharp=(JA,-JB)^\sharp  $
\end{enumerate}
with identical formulae for the `flat' tangent vectors.   (Recall that for the `sharp' tangent vectors, we are considering tangent vectors at $[g_u,W]$, where $W=-iP(Z)$.)  \qed
\end{lemma}


\section{Hyperk\"ahler structure}\label{s:hyper}

In this section, we deduce formulas for the metric and for the real symplectic forms $\omega_1$, $\omega_2$, $\omega_3$, and some additional facts.  The starting point is to identify $i(\omega_2+i\omega_3)$ with the KKS form on $\O$; it turns out that $\omega_1+i\omega_2$ gives the usual holomorphic symplectic form on $\twist GQ{\m^-}$.  

Using (\cite{H2}, Lemma 6.8), the hyperk\"ahler structure on $\O$ is established once the following has been verified:

\begin{prop}\label{prop:existence}

On $\O$, 
\begin{enumerate}
\item There exists a pair of  complex structures $(J_2,J_3)$ on $\O$ whose associated almost complex structures anticommute. 
\item There exist closed real 2-forms $\omega_2$, $\omega_3$ on $\O$ that are $J_2$- and $J_3$-invariant, respectively. 
\item For all  vectors $v$ and $w$ tangent to $\O$ at a point, we have 
$\omega_2(v,J_2w)=\omega_3(v,J_3w)$. 
\item The (real symmetric bilinear) form defined by $m(v,w)=\omega_2(v,J_2w)=\omega_3(v,J_3w)$
is positive-definite. 
\end{enumerate}
\end{prop}

\begin{proof}
(1) was verified in Lemma \ref{lemma:quaternionic}. For (2), {\it we define $i(\omega_2+i\omega_3)$ to be the complex Kostant-Kirillov-Souriau form on $\O$}, which   is the form (for $g\in G$ and $A_1,B_1,A_2,B_2\in\g_u)$ that at the point $\Ad_g\Upsilon\in\O$, sends the pair of tangent vectors 
$$\left(\left.\frac d{dt}\right|_0\Ad_{\exp(t(A_1+iB_1))}\Ad_g\Upsilon,\left.\frac d{dt}\right|_0\Ad_{\exp(t(A_2+iB_2))}\Ad_g\Upsilon\right)$$
to $\kappa\left(\Ad_g\Upsilon,[A_1+iB_1,A_2+iB_2]\right)$.  Thus $\omega_2$ (resp. $\omega_3$) is the imaginary part (resp. the opposite of the real part) of the KKS form.  Using the `tilde'  tangent vectors, we 
immediately obtain the expressions for $\omega_2$ and $\omega_3$ that appear in Proposition \ref{prop:metricintilde} below  
(using Lemma \ref{lemma:killingform}(1) to obtain the second version of each formula). Since the KKS form is closed, it follows that $\omega_2$ and $\omega_3$ are closed.  
Using 
 Lemmas \ref{lemma:order2}(2) and \ref{lemma:killingform}(3) and Corollary \ref{cor:order2again}, 
 it is easy to prove the required invariance of these forms.  Using the same tools along with  Lemma \ref{lemma:order2}(3), one can then verify 
 the equation in (3) and the formulas for $m$ in Proposition \ref{prop:metricintilde}.  The metric is nondegenerate and positive; this follows from the fact that the Killing form is negative-definite on $\g_u$, along with  Lemma \ref{lemma:killingform}(4) and Theorem \ref{theorem:eigenvaluetheorem}.  \end{proof}

It is straightforward to  obtain the following explicit formulae for $m$, $\omega_1$, $\omega_2$,  $\omega_3$, and certain holomorphic symplectic forms, given $g_u$, $Z$, and $W$ as earlier.
 The formulae involve the operators $E_Z$, $S_Z$, $S_{JZ}$; the effect of these operators is  easy to understand in light of Section \ref{s:eigen}.  Note that $\omega_1$ and $\omega_2$ are most easily described using the $\sharp$ or $\flat$ tangent vectors, whereas $\omega_3$ is more natural in the $\sim$ vectors.  

\begin{prop}\label{prop:metricinhat}
On the pair of tangent vectors $(A_1,B_1)\hat{}, (A_2,B_2)\hat{}$, we have 
\begin{itemize}
\item $m=-\kappa\left(\Upsilon,[S_ZA_1,S_ZJA_2]+[E_ZB_1,S_ZJS_Z^{-1}E_ZB_2]\right)  
      \newline\phantom{.}\quad  =-\kappa\left(S_{JZ}S_ZA_1,A_2\right)-\kappa\left(S_{JZ}S_Z^{-1}E_ZB_1,E_ZB_2\right)$
\item $\omega_1=
\kappa\left(\Upsilon,[E_ZB_1,S_ZJA_2]-[A_1,S_ZJE_ZB_2]\right)    
     \newline\phantom{.}\quad        =-\kappa\left(A_1,S_{JZ}E_ZB_2\right)+\kappa\left(S_{JZ}E_ZB_1,A_2\right)$
\item $\omega_2=\kappa\left(\Upsilon,[S_ZA_1,E_ZB_2]+[E_ZB_1,S_ZA_2]\right)   
           =\kappa\left(JS_ZA_1,E_ZB_2\right)-\kappa\left(E_ZB_1,JS_ZA_2\right)$
\item $\omega_3=\kappa\left(\Upsilon,[E_ZB_1,E_ZB_2]-[S_ZA_1,S_ZA_2]\right)  
          =\kappa\left(JE_ZB_1,E_ZB_2\right)-\kappa\left(JS_ZA_1,S_ZA_2\right)$
\item $\omega_2+i\omega_3= -i\kappa\left(\Upsilon, [S_ZA_1+iE_ZB_1,S_ZA_2+iE_ZB_2]\right)
        \newline\phantom{.}\quad         =-i\kappa\left(J(S_ZA_1+iE_ZB_1),S_ZA_2+iE_ZB_2\right)$ \qed
\end{itemize}
\end{prop}

\begin{prop}\label{prop:metricintilde}
On the pair of tangent vectors $(A_1,B_1)\tilde{}, (A_2,B_2)\tilde{}$, we have 
\begin{itemize}
\item $m= -\kappa\left(\Upsilon,[A_1,S_ZJS_Z^{-1}A_2]+[B_1,S_ZJS_Z^{-1}B_2]\right) 
  \newline\phantom{.}\quad      =-\kappa\left(S_{JZ}S_Z^{-1}A_1,A_2\right)-\kappa\left(S_{JZ}S_Z^{-1}B_1,B_2\right)  $
\item $\omega_1=\kappa\left(\Upsilon,[B_1,S_ZJS_Z^{-1}A_2]-[A_1,S_ZJS_Z^{-1}B_2]\right)    
      \newline\phantom{.}\quad       =\kappa\left(S_{JZ}S_Z^{-1}B_1,A_2\right)-\kappa\left(A_1,S_{JZ}S_Z^{-1}B_2\right)$
\item $\omega_2=\kappa\left(\Upsilon,[A_1,B_2]+[B_1,A_2]\right)   
          =\kappa\left(JA_1,B_2\right)-\kappa\left(B_1,JA_2\right)$
\item $\omega_3=\kappa\left(\Upsilon,[B_1,B_2]-[A_1,A_2]\right)  
         =\kappa\left(JB_1,B_2\right)-\kappa\left(JA_1,A_2\right)$
\item $\omega_2+i\omega_3=   -i\kappa\left(\Upsilon,[A_1+iB_1,A_2+iB_2]\right)   
         =-i\kappa\left(J(A_1+iB_1),A_2+iB_2\right) $  \qed
\end{itemize}
\end{prop}

\begin{prop}\label{prop:metricinsharp}
On the pair of tangent vectors $(A_1,B_1)^\sharp, (A_2,B_2)^\sharp$, we have  
\begin{itemize}
\item $m= -\kappa\left(\Upsilon,[S_ZA_1,S_ZJA_2]+[S_{JZ}^{-1}B_1,S_{JZ}^{-1}JB_2]\right)   
  \newline\phantom{.}\quad          =-\kappa\left(S_{JZ}S_ZA_1,A_2\right)-\kappa\left(S_{JZ}^{-1}S_Z^{-1}B_1,B_2\right)$
\item $\omega_1=
\kappa\left(\Upsilon,[B_1,JA_2]-[A_1,JB_2]\right)
         =-\kappa(A_1,B_2)+\kappa(B_1,A_2)$
\item $\omega_2= 
\kappa\left(\Upsilon,[A_1,B_2]+[B_1,A_2]\right)  
           =-\kappa(A_1,JB_2)+\kappa(JB_1,A_2)$
\item $\omega_3=\kappa\left(\Upsilon,[S_{JZ}^{-1}B_1,S_{JZ}^{-1}B_2]-[S_ZA_1,S_ZA_2]\right) 
     \newline\phantom{.}\quad       = \kappa\left(S_{JZ}^{-1}S_Z^{-1}JB_1,B_2\right)-\kappa\left(S_{JZ}S_ZJA_1,A_2\right)$
\item   $\omega_1+i\omega_2= 
\kappa\left(\Upsilon,[A_1,(iI-J)B_2]+[B_1,(iI+J)A_2]\right)  
 \newline\phantom{.}\quad        =-\kappa\left(A_1,B_2\right)+\kappa\left(B_1,A_2\right)-i\kappa\left(A_1,JB_2\right)+i\kappa\left(JB_1,A_2\right)$  \qed
\end{itemize}  
\end{prop}

The formulae in Proposition \ref{prop:metricinsharp} are also valid for the `flat' vector fields.  However, it is also convenient to express these formulae in terms of $A_1^+,A_1^-,B_2^+,B_2^-$.  Recall that by Lemma \ref{lemma:SZ}(3), $S_{JZ}\circ S_Z$ preserves $\m^+$.  

\begin{prop}\label{prop:metricinflat}
On the pair of tangent vectors $(A_1,B_1)^\flat, (A_2,B_2)^\flat$, we have  
\begin{itemize}
\item $m= -2\Re\left(\kappa\left(S_{JZ}S_ZA_1^+,A_2^-\right)+\kappa\left(S_{JZ}^{-1}S_Z^{-1}B_1^+,B_2^-\right)\right)      $
\item $\omega_1=2\Re\left(\kappa\left(B_1^+,A_2^-\right)-\kappa\left(A_1^+,B_2^-\right)\right)   $
\item $\omega_2=   -2\Im\left(\kappa\left(A_1^+,B_2^-\right)+\kappa\left(B_1^+,A_2^-\right)\right)$
\item $\omega_3= 2\Im\left(\kappa\left(S_{JZ}S_ZA_1^+,A_2^-\right)-\kappa\left(S_{JZ}^{-1}S_Z^{-1}B_1^+,B_2^-\right)\right)  $
\item  $\omega_1+i\omega_2= 2\kappa\left(B_1^-,A_2^+\right)-2\kappa\left(A_1^+,B_2^-\right)   $   \qed
\end{itemize}
\end{prop}

In the remainder of this section, we present some facts about the symplectic forms and K\"ahler potentials.

\begin{prop}\label{prop:canonicalforms}
Under the identification of $\twist{G}{Q}{\m^-}$ as the complex contangent bundle of $G/Q$, 
\begin{enumerate}
\item At $[g,W^-]\in\twist{G}{Q}{\m^-}$, the canonical 1-form $\lambda_\C$ on $\twist{G}{Q}{\m^-}$ sends $(A,B)^\flat$ to $2\kappa(A^+,W^-)$. 
\item The canonical (complex) 2-form $d\lambda_\C$ equals $\omega_1+i\omega_2$. 
\end{enumerate}
\end{prop}

\begin{proof}
Part (1) comes from the identification in Section \ref{ss:cotangenttwisted}.  For (2), we need to compute $d\lambda_\C$.  
To get around the fact that the `flat' tangent vectors do not give vector fields on $\twist{G}{Q}{\m^-}$, we work in $G\times\m^-$, whose quotient by $Q$ gives $\twist{G}{Q}{\m^-}$.  On $G\times \m^-$, for each $(A^+,B^-)\in \m^+\times\m^-$, we have a globally defined vector field $(A^+,B^-)^\heartsuit=\left.\frac{d}{dt}\right|_0\left(g\exp(tA^+),W^-+tB^-\right)$. We note that the bracket of two `heart' vector fields is zero; in the first component, we use that $[\m^+,\m^+]=0$, and the statement is obvious in the second component.  Also, note that 
for the   obvious lift of $\lambda_\C$ to $G\times \m^-$, we have
$(A_1^+,B_1^-)^\heartsuit \left(\lambda_\C(A_2^+,B_2^-)^\heartsuit\right)
=(A_1^+,B_1^-)^\heartsuit \left(2\kappa(A_2^+,W^-)\right)=2\kappa(A_2^+,B_1^-)$. 
Using the usual formula for the external derivative of a 1-form,\footnote{$d\lambda(X,Y)=X(\lambda(Y))-Y(\lambda(X))-\lambda([X,Y])$} we obtain 
$d\lambda_\C\left((A_1^+,B_1^-)^\heartsuit, (A_2^+,B_2^-)^\heartsuit\right)
=2\kappa(A_2^+,B_1^-)-2\kappa(A_1^+,B_2^-)$.  
Pushing this down to $\twist{G}{Q}{\m^-}$ and comparing with Proposition \ref{prop:metricinflat}, we obtain (2). 
\end{proof}

In a similar way, 
  
\begin{prop}\label{prop:canonicalforms2}
Under the identification of $\twist{G_u}{K_0}{\m_u}$ as the real contangent bundle of $G_u/K_0$, 
\begin{enumerate}
\item At $[g_u,W]\in\twist{G_u}{K_0}{\m_u}$, the canonical 1-form $\lambda_\R$ on $\twist{G_u}{K_0}{\m_u}$ sends $(A,B)^\sharp$ to $\kappa(A,W)$, and coincides with 
$\Re(\lambda_\C)$. 
\item The canonical (real) 2-form $d\lambda_\R$ equals $\omega_1$. \qed
\end{enumerate}
\end{prop}

We also regard $\lambda_\C$ and $\lambda_\R$ as $1$-forms on $\O$, via $\Phi$.

\begin{defn}\label{defn:phi}
Define $\phi:\O\rightarrow\R$, $\Ad_{g_u}\Ad_{\exp Z}\Upsilon\mapsto\kappa\left(\Ad_{\exp Z}\Upsilon, \Upsilon\right)$. 
\end{defn}

The following proposition, combined with Proposition \ref{prop:canonicalforms2}, shows that $\phi$ is a K\"ahler potential for $\omega_1$, relative to $J_1$, as proved in \cite{BG2}. 

\begin{prop}\label{prop:potential1}
$\lambda_\R=d_1^c\phi$, where $d_1^c$ denotes $d^c$ with respect to the complex structure $J_1$. 
\end{prop}

\begin{proof}
Let $A,B\in\m_u$.  We compare $\lambda_{\R}(A,B)\hat{}$ with $(d_1^c\phi)(A,B)\hat{}$. 

On the one hand, by Lemma \ref{lemma:relatevectorfields}, $(A,B)\hat{}$ at $\Ad_{g_u}\Ad_{\exp Z}\Upsilon$ corresponds under $\Phi$ to $(A,S_{JZ}E_ZB)^\sharp$ at $[g_u,-iP(Z)]$, which by Proposition \ref{prop:canonicalforms2}(1), is sent by $\lambda_\R$ to $\kappa(A,-iP(Z))$.  

On the other hand, by Lemma \ref{lemma:hatalmostcomplexstructures}(1), 
\begin{eqnarray*}(d_1^c\phi)(A,B)\hat{}&=&\left(J_1(A,B)\right)\hat{}(\phi)\\
&=&\left.\frac{d}{dt}\right|_0\kappa\left(\Ad_{\exp(Z+tiE_Z^{-1}S_ZA)}\Upsilon,
\Upsilon\right)\\&=&\left.\frac{d}{dt}\right|_0\kappa\left(\Ad_{\exp(tiE_{Z,r}E_Z^{-1}S_ZA)}\Ad_{\exp Z}\Upsilon,\Upsilon\right)\\&=&i\kappa\left([E_{Z,r}E_Z^{-1}S_ZA,\Ad_{\exp Z}\Upsilon],\Upsilon\right)=i\kappa\left(E_{Z,r}E_Z^{-1}S_ZA,[\Ad_{\exp Z}\Upsilon,\Upsilon]\right)\\&=&\kappa\left(E_{Z,r}E_Z^{-1}S_ZA,-iP(Z)\right),
\end{eqnarray*}
the last equality following from the definition of $P(Z)$.  
Since $\kappa(\k,\m)=0$, we may replace $E_{Z,r}$ with $E_Z$, yielding $\kappa\left(S_ZA,-iP(Z)\right)$.  By Lemma \ref{lemma:killingform}(3), this equals
$\kappa\left(A,-iS_ZP(Z)\right)$.  Taking $\a_0$ so that $Z\in\a_0$, we have that $P(Z)\in\a_0$ (Section \ref{ss:P}), and it follows that $\kappa\left(A,-iS_ZP(Z)\right)=\kappa\left(A,-iP(Z)\right)$ by Section \ref{s:eigen}. 
\end{proof}

In fact, $\phi$ is  related to another K\"ahler potential, viewed in the setting of $\twist{G_u}{K_0}{\m_u}$:

\begin{defn}\label{defn:phi'}
 Define $\phi':\twist{G_u}{K_0}{\m_u}\rightarrow\R$, 
$[g_u,W]\mapsto 2\kappa(P(\frac 12 P^{-1}(iW)), P(\frac 12 P^{-1}(iW)))$.  
\end{defn}

\begin{prop}\label{prop:potential2}
$\phi'$ is a K\"ahler potential for $\omega_1$, relative to $J_1$. 
\end{prop}

\begin{proof}

Beginning with the definition of $\phi$, we have 
$$\begin{aligned} \phi(\Ad_{g_u}\Ad_{\exp Z}\Upsilon)&=\kappa(\Ad_{\exp Z}\Upsilon,\Upsilon)=\kappa(\Ad_{\exp Z/2}\Upsilon,\Ad_{\exp -Z/2}\Upsilon)\\&=\kappa(X+Y,X-Y)=\kappa(X,X)-\kappa(Y,Y)\end{aligned}$$
where
$$\begin{aligned} X&=\text{pr}^{\g}_{\k}\Ad_{\exp Z/2}\Upsilon=\frac 12\left(\Ad_{\exp Z/2}+\Ad_{\exp -Z/2}\right)\Upsilon\in\k_0\\
Y&=\text{pr}^{\g}_{\m}\Ad_{\exp Z/2}\Upsilon=\frac 12\left(\Ad_{\exp Z/2}-\Ad_{\exp -Z/2}\right)\Upsilon\in\m_0\end{aligned}$$

But one computes easily that $\kappa(X,X)=\frac 12\kappa(\Upsilon,\Upsilon)+\frac 12\phi(\Ad_{g_u}\Ad_{\exp Z}\Upsilon)$.  Combining with the first equation in the proof, one concludes that $\phi(\Ad_{g_u}\Ad_{\exp Z}\Upsilon)
=-2\kappa(Y,Y)+\kappa(\Upsilon,\Upsilon)$.  

Recall that $Z=P^{-1}(iW)$.   We claim that $Y=-iP(\frac 12 P^{-1}(iW))$.  This is equivalent to $\text{pr}^{\g}_{\m}\Ad_{\exp Z/2}\Upsilon=-iP(\frac 12 Z)$, which follows directly from the definition of $P$.

Finally, $\phi(\Ad_{g_u}\Ad_{\exp Z}\Upsilon)=2\kappa(P(\frac 12 P^{-1}(iW)), P(\frac 12 P^{-1}(iW)))+\kappa(\Upsilon,\Upsilon)$, so $\phi$ and $\phi'$ differ by a constant.  
\end{proof}

\section{Moment maps}\label{s:moment} 

We give moment maps for the $G_u$-action on $\twist{G_u}{K_0}{\m_u}$ (Proposition \ref{prop:moment1}) and on $\O$ (Proposition \ref{prop:moment2}, relative to the symplectic structures $\omega_1$, $\omega_2$, and $\omega_3$.    The definition of a moment map is required in the proofs in \S\ref{s:proofs}, but the formulae for $\mu_1$, $\mu_2$, and $\mu_3$ derived in this section are not needed there.

\begin{defn}\label{defn:momentmap}
Let $(M,\omega)$ be a symplectic manifold and suppose that $H$ is a simple Lie group that acts symplectically on $M$.  A {\em moment map} is an $H$-equivariant map $\mu:M\rightarrow \h$ with the property that $$\kappa\left(d\mu_p(\chi_p),X\right)=\omega\left(\chi_p,\eta_X\right)$$ for all $p\in M$,  for all vectors $\chi_p$ tangent to $M$ at $p$, and for all $X\in\h$.  Here $\eta_X$ denotes the tangent vector (at $p$) induced by $X$ via the action of $H$.\footnote{It is customary to define moment maps as having target $\mathfrak h^*$, but if $H$ is simple, our definition is equivalent.} 
\end{defn}

\begin{prop}\label{prop:moment1} 
On $\twist{G_u}{K_0}{\m_u}$, with respect to the symplectic forms
$\omega_1$, $\omega_2$, and $\omega_3$, we have moment maps
$\mu_i:\twist{G_u}{K_0}{\m_u}\rightarrow\g_u$, given by 
\begin{eqnarray*}
\mu_1\left([g_u,W]\right)&=&\Ad_{g_u}W\\
\mu_2\left([g_u,W]\right)&=&\Ad_{g_u}JW\\
\mu_3\left([g_u,W]\right)&=&\frac 12\Ad_{g_u}\left(\Ad_{\exp\left(P^{-1}(iW)\right)}+
  \Ad_{\exp\left(-P^{-1}(iW)\right)} \right)\Upsilon
\end{eqnarray*}
\end{prop}

\begin{proof}
With the putative formula for $\mu_1$ and $A,B\in\m_u$, at $[g_u,W]$,  we have  that $d\mu_1\left(A,B\right)^\sharp{}=\Ad_{g_u}B+\Ad_{g_u}[A,W]$.  Let $X\in\g_u$; it is easy to see that at $[g_u,W]$, we have 
$\eta_{\Ad_{g_u}X}=\left(\proj{\g_u}{\m_u}X,[\proj{\g_u}{\k_0}X,W]\right)^\sharp{}$.  
We must verify the equation $\kappa\left(\Ad_{g_u}B+\Ad_{g_u}[A,W],\Ad_{g_u}X\right)=\omega_1\left((A,B)^\sharp{},(\proj{\g_u}{\m_u}X,[\proj{\g_u}{\k_0}X,W])^\sharp{}\right)$.  Using the fact that $\k$ and $\m$ are orthogonal under $\kappa$, the left side reduces to $\kappa(B,\proj{\g_u}{\m_u}X)+\kappa\left([A,W],\proj{\g_u}{\k_0}X\right)$.  So does the right side, using Proposition \ref{prop:metricinsharp} and the associativity of the Killing form.  The proof for $\mu_2$ is similar. 
The proof for $\mu_3$ will be given at the end of the proof of Proposition \ref{prop:moment2}.  
\end{proof}

\begin{cor}\label{cor:moment1}
The $\g$-valued, $G$-equivariant, $J_3$-holomorphic moment map on $\twist{G}{Q}{\m^-}$ relative to the holomorphic symplectic form $\omega_1+i\omega_2$, is given by 
$(\mu_1+i\mu_2)\left([g,W^-]\right)=2\Ad_gW^-$.
\end{cor}

\begin{proof}  Let $g=g_uq\in G$, with $g_u\in G_u$ and $q\in Q$. Then $[g,W^-]$ is mapped by 
$\Theta$ to $[g_u,(I+\theta)(\Ad_q(W^-))]$, which (Proposition \ref{prop:moment1}) is sent by $\mu_1+i\mu_2$ to 
\begin{eqnarray*}\Ad_{g_u}\left((I+iJ)(I+\theta)(\Ad_q(W^-))\right)
&=&\Ad_{g_u}\left(\right(I+iJ+\theta+iJ\theta)(\Ad_qW^-))\\
&=&\Ad_{g_u}\left(\right(I+iJ+\theta-\theta iJ)(\Ad_qW^-))\\
&=&2\Ad_{g_u}\Ad_qW^-=2\Ad_gW^-,
\end{eqnarray*}
using that $\Ad_qW^-\in\m^-$.
\end{proof}

\begin{prop}
\label{prop:moment2}
On $\O$, with respect to the symplectic forms $\omega_1$, $\omega_2$, and $\omega_3$, we have moment maps $\mu_i:\O\rightarrow\g_u$, given by 
\begin{eqnarray*}
\mu_1\left(\Ad_{g_u}\Ad_{\exp Z}\Upsilon\right)&=&\Im\left(\Ad_{g_u}\Ad_{\exp (JZ)}\Upsilon\right)\\
\mu_2\left(\Ad_{g_u}\Ad_{\exp Z}\Upsilon\right)&=&-\Im\left(\Ad_{g_u}\Ad_{\exp Z}\Upsilon\right),\text{ or otherwise said, }\\
 \mu_2\left(\Ad_g\Upsilon\right)&=&-\Im\left(\Ad_g\Upsilon\right)\\
 \mu_3\left(\Ad_{g_u}\Ad_{\exp Z}\Upsilon\right)&=&\Re\left(\Ad_g\Upsilon\right)
\end{eqnarray*}
\end{prop}

\begin{proof}
We first consider $\mu_2$.  We note that $\Ad_{g_u}\Ad_{\exp Z}\Upsilon$ is sent by $\Phi$ to $[g_u,-i\,P(Z)]$; Proposition \ref{prop:moment1} tells us that this is sent by $\mu_2$ to $-i\,\Ad_{g_u}JP(Z)=i\Ad_{g_u}\proj{\g}{i\g_u}\left(\Ad_{\exp Z}\Upsilon\right)
=-\Im\left(\Ad_{g_u}\Ad_{\exp Z}\Upsilon\right).$  The expression for $\mu_1$ is obtained in a similar way.   It remains to consider $\mu_3$.

With the putative formula  $\mu_3\left(\Ad_{g_u}\Ad_{\exp Z}\Upsilon\right)=\frac 12\Ad_{g_u}\left(\Ad_{\exp Z}+\Ad_{\exp(-Z)}\right)\Upsilon$, we have easily that 
$$\begin{aligned}d\mu_3\left((A,B)\hat{}\right)&=\frac 12\Ad_{g_u}[A,\left(\Ad_{\exp Z}+\Ad_{\exp(-Z)}\right)\Upsilon]\\&+\frac 12\Ad_{g_u}\left(\Ad_{\exp Z}[E_{Z}iB,\Upsilon]+\Ad_{\exp(-Z)}[E_{(-Z)}(-iB),\Upsilon]\right)\end{aligned}$$
which simplifies to 
$\frac 12\Ad_{g_u}[A,\left(\Ad_{\exp Z}+\Ad_{\exp(-Z)}\right)\Upsilon]-\frac i2\Ad_{g_u}\left(\Ad_{\exp Z}-\Ad_{\exp(-Z)}\right)JE_ZB$. 
It follows that for $X\in\g_u$, 
\begin{eqnarray*}
\kappa\left(d\mu_m(\chi_m),\Ad_{g_u}X\right)&=&\kappa\left(\frac 12\Ad_{g_u}[A,\left(\Ad_{\exp Z}+\Ad_{\exp(-Z)}\right)\Upsilon], \Ad_{g_u}X\right)\\&-&\kappa\left(\frac i2\Ad_{g_u}\left(\Ad_{\exp Z}-\Ad_{\exp(-Z)}\right)JE_ZB, \Ad_{g_u}X\right)\\&=&
\kappa\left(\frac 12[A,\left(\Ad_{\exp Z}+\Ad_{\exp(-Z)}\right)\Upsilon]-\frac i2\left(\Ad_{\exp Z}-\Ad_{\exp(-Z)}\right)JE_ZB, X\right)\\
\end{eqnarray*}

On the other hand, from 
\begin{eqnarray*}
\Ad_{\exp(t\Ad_{g_u}X)}\Ad_{g_u}\Ad_{\exp Z}\Upsilon
&=&\Ad_{g_u}\Ad_{\exp(tX)}\Ad_{\exp Z}\Upsilon\\
&\sim&\Ad_{g_u}\Ad_{\exp (t\,\proj{\g_u}{\m_u}X)}\Ad_{\exp(t\,\proj{\g_u}{\k_0}X)}\Ad_{\exp Z}\Upsilon\\
&\sim&
\Ad_{g_u}\Ad_{\exp (t\,\proj{\g_u}{\m_u}X)}\Ad_{\exp(Z+t[\proj{\g_u}{\k_0}X,Z])}\Upsilon
\end{eqnarray*}
we conclude that 
$\eta_{\Ad_{g_u}X}=\left(\proj{\g_u}{\m_u}X, -i[\proj{\g_u}{\k_0}X,Z]\right)\hat{}$.  
From Proposition \ref{prop:metricinhat} and Lemma \ref{lemma:killingform}, it follows that 
\begin{eqnarray*}
\omega_3\left((A,B)\hat{},\eta_{\Ad_{g_u}\proj{\g_u}{\m_u}X}\right)
&=&-\kappa\left(JS_ZA,S_Z\proj{\g_u}{\m_u}X\right)
+\kappa\left(JE_ZB,-i\,E_Z[\proj{\g_u}{\k_0}X,Z]  \right)\\
&=&-\kappa\left(S_ZJS_ZA,\proj{\g_u}{\m_u}X\right)
-i\,\kappa\left(\ad_ZE_ZJE_ZB,\proj{\g_u}{\k_0}X\right)\\
&=&-\kappa\left(S_ZJS_ZA,X\right)
-i\,\kappa\left(\ad_ZE_ZJE_ZB,X\right). 
\end{eqnarray*}
using the orthogonality of $\k_0$ and $\m_0$.

 To verify the formula for $\mu_3$, by Definition \ref{defn:momentmap}, we must verify that 
\begin{eqnarray*}
\frac 12[A,\left(\Ad_{\exp Z}+\Ad_{\exp(-Z)}\right)\Upsilon]&=&-S_ZJS_ZA\\
\frac 12\left(\Ad_{\exp Z}-\Ad_{\exp(-Z)}\right)JE_ZB&=&\ad_ZE_ZJE_ZB
\end{eqnarray*}

The first equation follows from Lemma \ref{lemma:identzero}(2),  
while the second equation follows from Lemma \ref{lemma:oddpart}.  

Finally, we obtain the version of $\mu_3$ in Proposition \ref{prop:moment1} using 
Lemmas \ref{lemma:describePhi}(4) and \ref{lemma:PandJcommute}(3).  
\end{proof}

   \begin{cor}\label{cor:moment2}
 The $\g$-valued, $G$-equivariant, $J_1$-holomorphic moment map on $\O$, relative to the holomorphic symplectic form $-i\omega_2+\omega_3$, is
$-i\mu_2+\mu_3$, which is simply the inclusion map. \qed
\end{cor}

\section{Statements of Main Results}
\label{s:deformations}
The following theorem shows that {\em for $\lambda\in\C^*$, $T_\lambda$ intertwines $J_{(\lambda)}$ with $J_1$.}  Thus, except for $\pm J_3$, all complex structures on $\O$ arising from the hyperk\"ahler structure are $G_u$-equivariantly equivalent to $J_1$ via an explicit map.  

\begin{theorem}\label{t:equiv} 
Let $\lambda\in\C^*$, let $A,B\in\m_u$, let $Z\in\m_0$, and let $g_u\in G_u$.  We have 
$$dT_{\lambda}^{-1}\circ J_1\circ dT_{\lambda}(A,B)\hat{}_{g_u,Z}=J_{(\lambda)}(A,B)\hat{}_{g_u,Z}  \qed $$

\end{theorem}

Recall that $J_1$ gives the usual complex structure on $\O$, and the usual holomorphic action of $G$ on $\O$ is $(g,x)\mapsto \Ad_g(x)$ (where $x\in \O$).  For $\lambda\in\C^*$ (which is to say, $J_{(\lambda)}\neq \pm J_3$), we have:

\begin{cor}\label{cor:J1}
\begin{enumerate}
\item
For $\O$ equipped with the complex structure $J_{(\lambda)}$, there is a holomorphic action of $G$ given by $(g,x)\mapsto 
T_{\lambda}^{-1}\circ \Ad_g\circ T_{\lambda}(x)$.  
\item This holomorphic action extends the standard action of $G_u$ on $\O$ via $\Ad$.  
\item This holomorphic action is transitive.\footnote{Note the change at $J_3$.  Here the holomorphic action is the usual action of $G$ on $\twist GQ{\m^-}$, which is not transitive.}  \qed
\end{enumerate}
\end{cor}

The map $T_\lambda$ also intertwines holomorphic symplectic forms for the complex structures $J_{(\lambda)}$ and $J_1$, as we now make precise.  
Beginning with $\lambda\in\C^*$, we can construct a holomorphic symplectic form relative to $J_{(\lambda)}$ as follows.  We use the notation from Section \ref{s:notation}.  Noting that the vector 
$(\lambda_1,\lambda_2,\lambda_3)$
 is orthogonal to the vector  $\frac 1{\lambda_1^2+\lambda_2^2}(-\lambda_2,\lambda_1,0)$ and that their cross product is 
$\frac 1{\lambda_1^2+\lambda_2^2}\left(-\lambda_1\lambda_3,-\lambda_2\lambda_3,\lambda_1^2+\lambda_2^2\right)$, we obtain the holomorphic symplectic form 
$$\frac{1}{\lambda_2^2+\lambda_2^2}\left(   (-\lambda_2-i\lambda_1\lambda_3)\omega_1  +   (\lambda_1-i\lambda_2\lambda_3)\omega_2  
+i(\lambda_1^2+\lambda_2^2)\omega_3
\right)$$
which simplifies to $\frac i2\left(\lambda-\lambda^{-1}\right)\omega_1+\frac 12\left(\lambda+\lambda^{-1}\right)\omega_2+i\omega_3$.  In fact, {\it this is the pullback  via the map $T_\lambda$  of the standard holomorphic symplectic form relative to $J_1$ }:

\begin{theorem}\label{t:holsymplintertwine}
$(dT_{\lambda})^*(\omega_2+i\omega_3)=\frac i2\left(\lambda-\lambda^{-1}\right)\omega_1+\frac 12\left(\lambda+\lambda^{-1}\right)\omega_2+i\omega_3$   \qed
\end{theorem}

\vskip .1 true in 

Above, we showed the equivalence between the complex structures $J_1$ and $J_{(\lambda)}\neq \pm J_3$, their associated holomorphic $G$ actions, and their associated holomorphic symplectic forms.  Below, we show that in contrast, the real symplectic forms associated to the various $J_{(\lambda)}$ are deformations of the real symplectic form $\omega_3$ associated to $J_3$, except when 
$|\lambda|=1$.  

Specifically, given $\lambda\in\C^*$ with $|\lambda|\neq 1$, define $\nu=\nu(\lambda)=\frac{\lambda(1+|\lambda|)}{|\lambda|(1-|\lambda|)}\in\C^*$.  
The following theorem shows that  {\it the pullback via $T_\nu$ of the standard real symplectic form for  $J_3$ is a nonzero constant multiple of the standard real symplectic form for $J_{(\lambda)}$: }  
\begin{theorem}\label{t:J3}
$\dfrac{1-|\lambda|^2}{1+|\lambda|^2}(dT_{\nu})^*(\omega_3)=\omega_{(J)}$   \qed 
\end{theorem}

It follows that {\em for each real symplectic form  $\omega_{(\lambda_1,\lambda_2,\lambda_3)}$ with $\lambda_3\neq 0$, the corresponding moment map is, up to scalar multiple, a deformation (via $T_{\nu}$) of the moment map $\mu_3$ associated with $\omega_3$:}

\begin{cor}\label{c:J3}
For all $\lambda\in\C^*$ with $\left|\lambda\right|\neq 1$, we have 
$\mu_{(\lambda)}=\dfrac{1-|\lambda|^2}{1+|\lambda|^2}\,\mu_3\circ T_\nu$.  \qed
\end{cor}   

Finally, we give explicit formulas for the function $T_\lambda$, both on $\mathcal O$, and (via $\Phi$) on $\twist{G_u}{K_0}{\m_u}$.  

\begin{theorem}\label{thm:Tlambdaformula}  
\begin{enumerate}
\item  Let $g_u\in G_u$ and let $Z\in \m_0$; without loss of generality we may assume that $Z=\sum c_\psi x_\psi\in\a_0$ for some $\{c_\psi\}\subset \R$.  Let  $\lambda\in\C^*$, and write $\lambda = \left|\lambda\right| u^2$ where $\left|u\right|=1$.  We then have that $T_\lambda:\mathcal O\rightarrow \mathcal O$ is given by $\Ad_{g_u}\Ad_{\exp Z}\Upsilon \mapsto \Ad_{g_u'}\Ad_{\exp Z'}\Upsilon$, where 
$$\begin{aligned}
g_u'&=g_u\cdot \Pi_{\psi}\exp(i \cdot s_\psi \cdot t_{u^2}(y_\psi))\\
Z'&=\sum_{\psi}r_\psi \cdot t_{u^2}(x_\psi)
\end{aligned}$$
and where 
$$\begin{aligned}
s_\psi &=\frac 12 \arctan\left(\frac 12\left(\dfrac 1{\left|\lambda\right|}-\left|\lambda\right|\right)\tanh(2c_\psi)\right)\\
r_\psi&=\frac 12\text{\rm arcsinh}\left(\frac 12\left(\dfrac 1{\left|\lambda\right|}+\left|\lambda\right|\right)\sinh(2c_\psi)\right)
\end{aligned}$$
\item Let $g_u\in G_u$ and let $W\in\m_u$; without loss of generality we may assume that $W=\sum id_\psi x_\psi\in i\a_0$ for some $\{d_\psi\}\subset\R$.  Take $\lambda$ as above.  We then have that $\Phi\circ T_\lambda \circ \Phi^{-1}: \twist{G_u}{K_0}{\m_u}\rightarrow \twist{G_u}{K_0}{\m_u}$ is given by $[g_u,W]\mapsto [g_u',W']$, where 
$$\begin{aligned}
g_u'&=g_u\cdot \Pi_{\psi}\exp(i \cdot s_\psi \cdot t_{u^2}(y_\psi))\\
W'&=\frac i2\left(\frac 1{\left|\lambda\right|}+\left|\lambda\right|\right)\sum_{\psi} d_\psi\cdot t_{u^2}(x_\psi)=\frac 12\left(t_{\lambda}+t_{1/\overline{\lambda}}\right)W
\end{aligned}$$
and where
$$\begin{aligned}
s_\psi&=\frac 12\arctan\left(\left(\frac 1{\left|\lambda\right|}-\left|\lambda\right|\right)\left(\frac{-d_\psi}{1+4 d_{\psi}^2}\right)\right)\\
\end{aligned}\qed$$ 
\end{enumerate}
\end{theorem}

Note that for $\lambda$ as written, the role of $u$ is relatively uninteresting, compared to the situation when $\left|\lambda\right|\neq 1$.  Note also the remarkably simple formula $W'=\frac 12\left(t_{\lambda}+t_{1/\overline{\lambda}}\right)W$.

\begin{cor}\label{cor:Tlambdabijection} 
$T_\lambda$ is a bijection.  \qed
\end{cor}

\section{Proofs of Main Results}\label{s:proofs}

\subsection{Proof of Theorem \ref{t:equiv}}

\begin{prop}\label{p:ABCD}
Let $g\in G$, and $X,A_0,B_0,C,D\in\m$.  The following are equivalent:
\begin{enumerate}
\item $\left.\frac d{dt}\right|_0\Ad_g \Ad_{\exp tC}\Ad_{\exp (X+itD)}\Upsilon
= \left.\frac d{dt}\right|_0  \Ad_g\Ad_{\exp tiA_0}\Ad_{\exp(X-tB_0)}\Upsilon$
\item $S_XC+iE_XD=iS_XA_0-E_XB_0$
\end{enumerate}\end{prop}

\begin{proof}
$$\begin{aligned}\label{*}
1.&\Leftrightarrow \Ad_{\exp (X+itD)}\Upsilon\sim\Ad_{\exp t(iA_0-C)}\Ad_{\exp (X-tB_0)}\Upsilon\\
&\Leftrightarrow \Ad_{\exp X}\Ad_{\exp t\,E_{X,r}(iD)}\Upsilon\sim\Ad_{\exp t(iA_0-C)}\Ad_{\exp X}\Ad_{\exp t\,E_{X,r}(-B_0)}\Upsilon\\
&\Leftrightarrow \Ad_{\exp X}\Ad_{\exp t\,E_{X,r}(iD)}\Upsilon\sim\Ad_{\exp X}\Ad_{\exp t(\Ad_{\exp -X}(iA_0-C)-E_{X,r}(B_0))}\Upsilon\\
&\Leftrightarrow E_{X,r}(iD)\equiv \Ad_{\exp -X}(iA_0-C)-E_{X,r}(B_0)\quad\text{mod $\k$}\\
&\Leftrightarrow E_X(iD)=S_X(iA_0-C)-E_X(B_0)\Leftrightarrow 2.
\end{aligned}$$
\end{proof}

We proceed to the proof of Theorem \ref{t:equiv}.  Let $g_0\in G_u$, $Z\in \m_0$, and $A,B\in \m_u$.  We compute:

$$\begin{aligned}
(A,B)\hat{}_{g_u,Z}&=\left.\frac d{dt}\right|_0\Ad_{g_u}\Ad_{\exp tA}\Ad_{\exp (Z+itB)}\Upsilon\\&\underset{dT_{\lambda}}\mapsto
\left.\frac d{dt}\right|_0\Ad_{g_u}\Ad_{\exp tA}\Ad_{\exp (t_\lambda Z+tit_{\lambda}B)}\Upsilon\\
&\underset{J_1}\mapsto\left.\frac d{dt}\right|_0
\Ad_{g_u}\Ad_{\exp tiA}\Ad_{\exp (t_\lambda Z-tt_\lambda B)}\Upsilon\\
\end{aligned}$$
\noindent We  apply Proposition \ref{p:ABCD} with $X=t_\lambda Z$, $A_0=A$, and $B_0=t_\lambda B$;  this determines $C,D$ by the equation
$S_{t_\lambda Z}C+iE_{t_\lambda Z}D=iS_{t_\lambda Z}A-E_{t_\lambda Z}t_\lambda B$ with  the additional assumption that 
$C$ and $t_\lambda^{-1}D$ are in $\m_u$.  With these assumptions, the displayed calculation above continues as
$$\begin{aligned}
&=\left.\frac d{dt}\right|_0 \Ad_{g_u}\Ad_{\exp tC}\Ad_{\exp t_\lambda Z+itt_\lambda(t_{\lambda}^{-1}D)}\Upsilon\\
&\underset{dT_{\lambda}^{-1}}\mapsto \left.\frac d{dt}\right|_0 \Ad_{g_u}\Ad_{\exp tC}\Ad_{\exp  Z+itD}\Upsilon\\
&=(C,t_\lambda^{-1}D)\hat{}_{g_u,Z}
\end{aligned}$$

Comparing with Definition \ref{d:stereo}, we see that the proof of Theorem \ref{t:equiv} is complete after we have verified the following proposition:

\begin{prop}\label{p:ABCD2}
With notation as above, 
$$\begin{aligned}
C&=\lambda_1(-S_Z^{-1}E_ZB)+\lambda_2(-JS_Z^{-1}E_ZB)+\lambda_3JA\\
t_{\lambda}^{-1}D&=\lambda_1(E_Z^{-1}S_ZA)+\lambda_2(-E_Z^{-1}S_ZJA)+\lambda_3(-E_Z^{-1}S_ZJS_Z^{-1}E_ZB)
\end{aligned}$$
\end{prop}

\begin{proof}
So defined, we would have that $C$ and $t_{\lambda}^{-1}D$ are in $\m_u$, so the only question is whether the putative formulas for $C$ and $D$
satisfy the equation $S_{t_\lambda Z}C+iE_{t_\lambda Z}D=iS_{t_\lambda Z}A-E_{t_\lambda Z}t_\lambda B$.  Substituting the putative formulas for $C$ and $D$ into the equation and separating the conditions on $A$ and $B$, it is sufficient to verify that 
$$\begin{aligned}
i t_{\lambda}S_{t_{\lambda}Z}(A)&=i\left(\lambda_1E_{t_{\lambda}Z}t_\lambda E_Z^{-1}S_Z-\lambda_2 E_{t_{\lambda}Z}t_{\lambda}E_Z^{-1}S_ZJ-i\lambda_3 S_{t_{\lambda}Z}J\right)A\\
-E_{t_{\lambda}Z}t_\lambda B &=-\left( \lambda_1 S_{t_{\lambda}Z}S_Z^{-1}E_Z+\lambda_2 S_{t_{\lambda}Z}JS_Z^{-1}E_ZB+i\lambda_3 E_{t_{\lambda}Z}t_{\lambda}E_Z^{-1}S_ZJS_Z^{-1}E_Z \right)B
\end{aligned}$$
and after applying Lemma \ref{l:t}(4) and simplifying, this reduces to 
$$\begin{aligned}t_{\lambda}^{-1}&=\lambda_1-\lambda_2J-i\lambda_3 t_{\lambda}^{-1}J\\
I&=\lambda_1t_{\lambda}^{-1}+\lambda_2 t_{\lambda}^{-1}J+i\lambda_3 J
\end{aligned}$$
The first of these is  equivalent to the equation that was verified in Lemma \ref{l:t}(6).  The second is similar, but with the substitution 
$\lambda\mapsto 1/\lambda$, $\lambda_1\mapsto \lambda_1$, $\lambda_2\mapsto -\lambda_2$, $\lambda_3\mapsto -\lambda_3$.  
\end{proof}

\subsection{Proof of Corollary \ref{cor:J1}}

Part 1 is immediate from Theorem \ref{t:equiv}, which gives Part 3.   Part 2 is a consequence of the fact that $T_{\lambda}$ is $G_u$-equivariant (Definition \ref{d:T}). \qed

\subsection{Proof of Theorem \ref{t:holsymplintertwine}}
We use the ``hat'' tangent vectors, with $g_u\in G_u$ and $Z\in\m_0$ fixed, and with $A,B\in\m_u$.  We calculate $dT_\lambda$:
$$\begin{aligned}
(A,B){\hat{}}_{g_u,Z}&\underset{dT_{\lambda}}\longmapsto \left.\frac d{dt}\right|_0 \Ad_{g_u}\Ad_{\exp(tA)}\Ad_{\exp(t_\lambda Z+itt_{\lambda}B)}\Upsilon\\
&=\left.\frac d{dt}\right|_0  \Ad_{g_u}\Ad_{\exp(tA)}\Ad_{\exp(t_\lambda Z)}\Ad_{\exp(itE_{t_\lambda Z}(t_\lambda B))}\Upsilon   \\
&=\left.\frac d{dt}\right|_0 \Ad_{g_u}\Ad_{\exp(t_\lambda Z)}\Ad_{\exp(t\Ad_{\exp(-t_\lambda Z)}(A))} \Ad_{\exp(itE_{t_\lambda Z}(t_\lambda B))}\Upsilon  \\
&=\left.\frac d{dt}\right|_0  \Ad_{g_u}\Ad_{\exp(t_\lambda Z)}\Ad_{\exp(
t(\Ad_{\exp(-t_\lambda Z)}(A)+itE_{t_\lambda Z}(t_\lambda B))
)}\Upsilon  \\
\end{aligned}$$
Then, from the definition of $\omega_2+i\omega_3$ as $-i$ times the KKS form, we have, on the one hand, that 
$$\begin{aligned}
(\omega_2+i\omega_3)&\left(dT_{\lambda}(A_1,B_1)\hat{},dT_{\lambda}(A_2,B_2)\hat{}\right)\\
&=-i\kappa\left(\Upsilon,[\Ad_{\exp(-t_\lambda Z)}A_1+iE_{t_\lambda Z}t_\lambda B_1,\Ad_{\exp(-t_\lambda Z)}A_2+iE_{t_\lambda Z}t_\lambda B_2]\right)\\
&=-i\kappa\left([\Upsilon,\Ad_{\exp(-t_\lambda Z)}A_1+iE_{t_\lambda Z}t_\lambda B_1],\Ad_{\exp(-t_\lambda Z)}A_2+iE_{t_\lambda Z}t_\lambda B_2\right)\\
&=-i\kappa\left(JS_{t_\lambda Z}A_1+iJE_{t_\lambda Z}t_\lambda B_1,S_{t_\lambda Z}A_2+iE_{t_\lambda Z}t_\lambda B_2\right)\\
&= -i\kappa\left(JS_{t_\lambda Z}A_1+it_\lambda JE_{Z} B_1,S_{t_\lambda Z}A_2+it_\lambda E_{Z}B_2\right)  \qquad{\text{(Lemma \ref{l:t}(5)})} \\
&=-i\kappa(JS_{t_\lambda Z}A_1, S_{t_\lambda Z}A_2)  +i\kappa(t_\lambda JE_ZB_1,t_\lambda E_ZB_2)\\
     &\qquad+\kappa(t_\lambda JE_ZB_1, S_{t_\lambda Z}A_2)+\kappa(JS_{t_\lambda Z}A_1,t_\lambda E_ZB_2)\\
&=  -i\kappa(JS_{t_\lambda Z}A_1, S_{t_\lambda Z}A_2)  +i\kappa(t_\lambda JE_ZB_1,t_\lambda E_ZB_2)\\
  &\qquad+\kappa(t_\lambda JE_ZB_1, t_\lambda S_{Z}t_{\lambda}^{-1}A_2)+\kappa(Jt_\lambda S_{Z}t_\lambda^{-1}A_1,t_\lambda E_ZB_2)\\
 &=  -i\kappa(JS_{t_\lambda Z}A_1, S_{t_\lambda Z}A_2)  +i\kappa(JE_ZB_1,E_ZB_2)\\
  &\qquad-\kappa(E_ZB_1,  JS_{Z}t_{\lambda}^{-1}A_2)+\kappa(JS_{Z}t_\lambda^{-1}A_1, E_ZB_2),\\
\end{aligned}$$
 the last equality via Lemmas \ref{lemma:killingform}(5) and \ref{lemma:killingform}(2).  
On the other hand, we can use Proposition \ref{prop:metricinhat} to compute that 
$\frac i2\left(\lambda-\lambda^{-1}\right)\omega_1+\frac 12\left(\lambda+\lambda^{-1}\right)\omega_2+i\omega_3$ evaluates on our pair of ``hat'' vectors to 
$$\begin{aligned}\frac i2(\lambda-\lambda^{-1})\left(-\kappa(A_1,S_{JZ}E_ZB_2)\right.&\left.+\kappa(S_{JZ}E_ZB_1,A_2)\right)\\
&+\frac 12(\lambda+\lambda^{-1})\left(\kappa(JS_ZA_1,E_ZB_2)-\kappa(E_ZB_1,JS_ZA_2)\right)\\
&+i\left(\kappa(JE_ZB_1,E_ZB_2)-\kappa(JS_ZA_1,S_ZA_2)\right)\end{aligned}$$
After collecting terms and using Lemmas \ref{lemma:order2}(2) and \ref{lemma:killingform}(3) and the projection formulas in Lemma \ref{lemma:projectionlemma}(2,3), the expression simplifies to 
$$\begin{aligned}
-i\kappa(JS_ZA_1,S_ZA_2)&+i\kappa(JE_ZB_1,E_ZB_2)\\
&+\kappa(JS_Z\circ(\lambda\,\proj{\m}{\m^-}+\lambda^{-1}\,\proj{\m}{\m^+})A_1,E_ZB_2)\\
&+\kappa(E_ZB_1, JS_Z(-\lambda\,\proj{\m}{\m^-}-\lambda^{-1}\proj{\m}{\m^+})A_2)
\end{aligned}$$
To prove Theorem \ref{t:holsymplintertwine}, we compare like terms. 
It is immediate that the corresponding $B_1,B_2$ terms are equal. 
The corresponding $A_1,B_2$ and $B_1,A_2$ terms are equal, simply by interpreting $t_{\lambda}^{-1}$ as $\lambda\,\proj{\m}{\m^-}+\lambda^{-1}\,\proj{\m}{\m^+}$.    For the $A_1,A_2$ terms, we must show that $\kappa(JS_{t_\lambda Z}A_1,S_{t_\lambda Z}A_2)=\kappa(JS_ZA_1,S_ZA_2)$.  Using Lemmas \ref{lemma:order2}(2) and \ref{lemma:killingform}(3), this is equivalent to $\kappa(JS_{Jt_{\lambda}Z}S_{t_{\lambda}Z}A_1,A_2)=\kappa(JS_{JZ}S_ZA_1A_2)$.  By Lemma \ref{lemma:projectioncommutes}, the operators 
$S_{Jt_{\lambda}Z}S_{t_{\lambda}Z}$ and $S_{JZ}S_Z$ map $\m^+$ to $\m^+$ and $\m^-$ to $\m^-$; then by Lemma \ref{l:t}(5) we have that $S_{Jt_\lambda Z}S_{t_\lambda Z}=t_\lambda S_{JZ}S_Z t_\lambda^{-1}$.  
Using that $\kappa$ restricted to either $\m^+$ or $\m^-$ is zero,  we have that 
$$\begin{aligned}\kappa(JS_{Jt_{\lambda}Z}S_{t_{\lambda}Z}A_1,A_2)&=\kappa(JS_{Jt_{\lambda}Z}S_{t_{\lambda}Z}A_1^+,A_2^-)
+\kappa(JS_{Jt_{\lambda}Z}S_{t_{\lambda}Z}A_1^-,A_2^+)\\
&=\kappa(Jt_\lambda S_{JZ}S_Zt_\lambda^{-1}A_1^+,A_2^-)+\kappa(Jt_\lambda S_{JZ}S_Zt_\lambda^{-1}A_1^-,A_2^+)\\
&=\kappa(i\lambda S_{JZ}S_Z \lambda^{-1}A_1^+,A_2^-)+\kappa(-i\lambda^{-1}S_{JZ}S_Z\lambda A_1^-,A_2^+)\\
&=\kappa(iS_{JZ}S_ZA_1^+,A_2^-)+\kappa(-iS_{JZ}S_ZA_1^-,A_2^+)\end{aligned}$$
Since this is independent of the variable $\lambda$, it coincides with the value when $\lambda=1$, which is $\kappa(JS_{JZ}S_ZA_1,A_2)$.  \qed

We record a result of our computation:
\begin{prop}\label{prop:holsympcalc}
$$\begin{aligned}(\omega_2+i\omega_3)\left(dT_{\lambda}(A_1,B_1)\hat{},dT_{\lambda}(A_2,B_2)\hat{}\right)&=
-i\kappa(JS_ZA_1,S_ZA_2)+i\kappa(JE_ZB_1,E_ZB_2)\\
&+\kappa(JS_Zt_\lambda^{-1}A_1,E_ZB_2)
-\kappa(E_ZB_1, JS_Zt_\lambda^{-1}A_2)\qed
\end{aligned}
$$
\end{prop}

\subsection{Proof of Theorem \ref{t:J3}}
On the one hand, since $\omega_3$ is the imaginary part of $\omega_2+i\omega_3$, we can apply Proposition \ref{prop:holsympcalc}, with help from Lemma \ref{l:t}(7), to conclude that 
$$\begin{aligned}(\omega_3)\left(dT_{\nu}(A_1,B_1)\hat{},dT_{\nu}(A_2,B_2)\hat{}\right)&=
-\kappa(JS_ZA_1,S_ZA_2)+\kappa(JE_ZB_1,E_ZB_2)\\
&+\kappa(JS_Z  \circ \frac{t_\nu^{-1}-t_{\overline\nu}}{2i}(A_1),E_ZB_2)\\
&-\kappa(E_ZB_1, JS_Z \circ  \frac{t_\nu^{-1}-t_{\overline\nu}}{2i}   (A_2))
\end{aligned}
$$
On the other hand, $\omega_{(J)}=\lambda_1\omega_1+\lambda_2\omega_2+\lambda_3\omega_3$, so applying Proposition \ref{prop:metricinhat}, we have that 
$$
\begin{aligned}\omega_{(J)}\left((A_1,B_1)\hat{},(A_2,B_2)\hat{}\right)&=\lambda_1\left(-\kappa(A_1,S_{JZ}E_ZB_2)+\kappa(S_{JZ}E_ZB_1,A_2)\right)\\
&+\lambda_2\left(\kappa(JS_ZA_1,E_ZB_2)-\kappa(E_ZB_1,JS_ZA_2)\right)\\&+\lambda_3\left(\kappa(JE_ZB_1,E_ZB_2)-\kappa(JS_ZA_1,S_ZA_2)\right)
\end{aligned}$$
which, by familiar techniques, simplifies to 
$$\begin{aligned}
-\lambda_3&\kappa(JS_ZA_1,S_ZA_2)+\lambda_3\kappa(JE_ZB_1,B_2)\\&+\kappa(JS_Z\circ(\lambda_1J+\lambda_2)A_1,E_ZB_2)
-\kappa(E_ZB_1,JS_Z\circ(\lambda_1J+\lambda_2)A_2)
\end{aligned}$$
We have immediately the desired result with regard to the $A_1,A_2$ and $B_1,B_2$ terms.  To finish the proof, we need
${t_\nu^{-1}-t_{\overline\nu}} =2i \frac{\lambda_1J+\lambda_2}{\lambda_3}$ as operators on $\m$.  This is easily verified;  both operators are multiplication by  $\frac{-4\overline\lambda}{1-|\lambda|^2}$  on $\m^+$ and multiplication by  $\frac{4\lambda}{1-|\lambda|^2}$   on $\m^-$.  
\qed

\subsection{Proof of Corollary \ref{c:J3}}
First, the putative formula for $\mu_{(\lambda)}$ is $G_u$-equivariant since $T_{\nu}$ and $\mu_3$ both are $G_u$-equivariant.  

Second, we must verify the key formula for moment maps, as stated in Definition \ref{defn:momentmap}.  We let $p\in\O$, let $\chi_p$ be a vector tangent to $\O$ at $p$, let $X\in\g_u$, and let $(\eta_X)_p=\left.\frac{d}{dt}\right|_0 \Ad_{\exp(tX)}p$.  We must show
\begin{equation}
\frac{1-|\lambda|^2}{1+|\lambda|^2}\,\kappa\left(d(\mu_3\circ T_{\nu})(\chi_p),X\right)=\omega_{(\lambda)}\left(\chi_p,\left(\eta_X\right)_p\right)
\label{eq:moment1}
\end{equation}
Since $\omega_{(\lambda)}=\dfrac{1-|\lambda|^2}{1+|\lambda|^2}\,dT_{\nu}^*(\omega_3)$ by Theorem \ref{t:J3}, we see that Equation \ref{eq:moment1} is equivalent to 
\begin{equation}
\kappa\left(d(\mu_3 \circ T_{\nu})(\chi_p),X\right)=dT_{\nu}^*\omega_3\left(\chi_p,(\eta_X)_p\right)
\label{eq:moment2}
\end{equation}
The left side of (\ref{eq:moment2}) is $\kappa\left(d\mu_3(dT_{\nu}(\chi_p)),X\right)$.  The right side of (\ref{eq:moment2}) is 
$\omega_3\left(dT_{\nu}(\chi_p), dT_{\nu}(\eta_X)_p\right)$.  However, $dT_{\nu}(\eta_X)_p=\left(\eta_X\right)_{T_{\nu}(p)}$.  We see this since, on the one hand 
\begin{equation}  dT_{\nu}(\eta_X)_p=\left.\frac{d}{dt}\right|_0T_{\nu}\left(\Ad_{\exp(tX)} p\right),
\end{equation}
and on the other hand, 
\begin{equation}
\left(\eta_X\right)_{T_{\nu}(p)}=\left.\frac{d}{dt}\right|_0\Ad_{\exp(tX)}\left(T_\nu(p)\right);
\end{equation}
these are equal since $T_{\nu}$ is  $G_u$-equivariant.  
With these substitutions, Equation \ref{eq:moment2} becomes
$$\kappa\left(d\mu_3(dT_{\nu}(\chi_p)),X\right)
=\omega_3\left(dT_{\nu}(\chi_p),\left(\eta_X\right)_{T_{\nu}(p)}\right)$$ which is valid since $\mu_3$ is the moment map relative to $\omega_3$.  
\qed

\subsection{Proof of Theorem \ref{thm:Tlambdaformula}}
We begin with $\mathcal O$.  By strong orthogonality, the proof reduces to a calculation in $SL(2)$, namely, 
\begin{equation}\label{eq:Tlambda}
\Ad_{\exp\left(\begin{matrix}0&c\lambda\\c/\lambda&0\end{matrix}\right)}\left(\begin{matrix}1&0\\0&-1\end{matrix}\right)=
\Ad_{\exp\left(\begin{matrix}0&-su^2\\s\overline u^2&0\end{matrix}\right)}\Ad_{\exp\left(\begin{matrix}0&ru^2\\r\overline u^2&0\end{matrix}\right)}\left(\begin{matrix}1&0\\0&-1\end{matrix}\right)
\end{equation} 
where $r$ and $s$ are as in the statement of the theorem.  

On the one hand, by Lemma \ref{lemma:sl2}, the left side of Equation \ref{eq:Tlambda} is $\left(\begin{matrix}\cosh 2c&-\lambda\sinh 2c\\\frac{\sinh 2c}{\lambda}&-\cosh 2c\end{matrix}\right)$.  Using Lemma \ref{lemma:sl2}, matrix multiplication, and trigonometric identities, one can show that the right side of Equation \ref{eq:Tlambda} is 
$$\left(\begin{matrix} \cosh(2r)\cos(2s)&u^2\cosh(2r)\sin(2s)-u^2\sinh(2r)\\\overline u^2\cosh(2r)\sin(2s)+\overline u^2\sinh(2r)
&-\cosh(2r)\cos(2s)\end{matrix}\right)$$
We are finished if we can verify the equations
\begin{itemize}
\item[(a)]  $\cosh(2c)=\cosh(2r)\cos(2s)$
\item[(b)]  $\left|\lambda\right|\sinh(2c)=\sinh(2r)-\cosh(2r)\sin(2s)$
\item[(c)]  $\frac{\sinh(2c)}{\left|\lambda\right|}=\sinh(2r)+\cosh(2r)\sin(2s)$
\end{itemize}
which is to say
\begin{itemize}
\item [(a)]  $\cosh(2c)=\cosh(2r)\cos(2s)$
\item[(b)]  $\frac 12\left(\frac 1{\left|\lambda\right|}+\left|\lambda\right|\right)\sinh(2c)=\sinh(2r)$
\item[(c)]  $\frac 12\left(\frac 1{\left|\lambda\right|}-\left|\lambda\right|\right)\sinh(2c)=\cosh(2r)\sin(2s)$
\end{itemize}
Equation (b) follows immediately from the definition of $r$.  Dividing the corresponding sides of (c) and (a), one obtains the equation 
$\frac 12\left(\frac 1{\left|\lambda\right|}-\left|\lambda\right|\right)\tanh(2c)=\tan(2s)$, which is true by the definition of $s$.  We conclude by verifying Equation (a).  Note that since $\cosh$ takes on positive values and since $\cos(2s)$ is positive by the definition of $s$,  it is sufficient to verify that (a) $\cosh^2(2c)=\cosh^2(2r)\cos^2(2s)$.  We can write $\cosh^2(2r)$ as $1+\sinh^2(2r)$ and $\cos^2(2s)$ as $\frac{1}{1+\tan^2(2s)}$ and then use the definitions of $r$ and $s$ to simplify $\cosh^2(2r)\cos^2(2s)$; after algebraic manipulation and the use of the Pythagorean trigonometric identities, this reduces to $\cosh^2(2c)$.  This concludes the proof of the first part of Theorem \ref{thm:Tlambdaformula}.  

In the second part of Theorem \ref{thm:Tlambdaformula}, we observe that by the formula in Section \ref{ss:P}, $\Phi^{-1}[g_u,W]
=\Ad_{g_u}\Ad_{\exp(\sum c_\psi x_\psi)}\Upsilon$, where $c_\psi=\frac 12\text{\rm arcsinh}(-2d_\psi)$.  This is sent by $T_\lambda$ to $\Ad_{g_u'}\Ad_{\exp Z'}\Upsilon$, where $g_u'$ and $Z'$ are as in Part 1 of the theorem.  We consider these in turn.  First, 
in the formula for $g_u'$, we note that 
$$\begin{aligned}
s_\psi&=\frac 12\arctan\left(\frac 12\left(\frac 1{\left|\lambda\right|}-\left|\lambda\right|\right)\tanh(2c_\psi)\right)\\
&=\frac 12\arctan\left(\frac 12\left(\frac 1{\left|\lambda\right|}-\left|\lambda\right|\right)\tanh(\text{\rm arcsinh}(-2d_\psi))\right)\\
&=\frac 12\arctan\left(\left(\frac 1{\left|\lambda\right|}-\left|\lambda\right|\right)\frac{-d_\psi}{1+4d_{\psi}^2}\right)\\
\end{aligned}$$
using that $\tanh(\text{\rm arcsinh} z)=\frac z{\sqrt{1+z^2}}$.  Second, we know that $Z'=\sum_{\psi}r_\psi \cdot t_{u^2}(x_\psi)$; then note that in computing $\Phi\left(\Ad_{g_u'}\Ad_{\exp Z'}\Upsilon\right)$, in the second slot the entry $W'$ will be 
$$
\begin{aligned}-iP\left(\sum_{\psi}r_\psi \cdot t_{u^2}(x_\psi)\right)&=-i\cdot \frac 12   \sum_{\psi}\sinh(2r_\psi) \cdot t_{u^2}(x_\psi)\\
&=-\frac i2\sum_{\psi}-\left(\frac 1{\left|\lambda\right|}+\left|\lambda\right|\right)d_\psi\cdot t_{u^2}(x_\psi)\\
&=\frac i2\left(\frac 1{\left|\lambda\right|}+\left|\lambda\right|\right)\sum_{\psi}d_\psi\cdot t_{u^2}(x_\psi)\\
\end{aligned}$$
The formula $W'=\frac 12\left(t_{\lambda}+t_{1/\overline{\lambda}}\right)W$ follows immediately.  \qed

\subsection{Proof of Corollary \ref{cor:Tlambdabijection}}
This is easily seen in the setting $\twist{G_u}{K_0}{\m_u}$. The key observation is that $\frac 12\left(t_\lambda +t_{1/\overline\lambda}\right):\m_u\rightarrow\m_u$ is bijective. Note that by Theorem \ref{thm:Tlambdaformula}, 
$T_\lambda$ sends $[g_u,W]$ to $[g_u', \frac 12(t_\lambda +t_{1/\overline\lambda})W]$, where $g_u'$ equals $g_u$ times an element of $G_u$ that depends on $W$.  For this, it is an easy exercise to see that $T_\lambda$ is a bijection.        \qed

\section{Application:  a new view of an old action}
\label{s:example}
In this section, we apply the main results of \S\ref{s:deformations}, focusing on the action of $G_0$ on $\O$.

On the one hand, recall that for each $\lambda\in\C^*$, there is a holomorphic action of $G$ on $\O$, 
relative to the complex structure $J_{(\lambda)}$.  By Corollary \ref{cor:J1}, each of these actions is equivalent to the usual action of $G$ on $\O$ via the explicit intertwining map $T_{\lambda}$.  For $\lambda=0$ or $\infty$, the  
the holomorphic action degenerates to the usual holomorphic action of $G$ on $T^*\left(G/Q\right)$.  

On the other hand, as developed in \cite{HS}, \cite{HHS}, for any choice of $\lambda$, there is a   {\em moment-critical subset} of $\O$ (or equally well, of $T^*\left(G/Q\right)$), which arises via the composition
$$\O\overset{\mu_{(\lambda)}}\to\g_u\overset{\text{proj}}\to\m_u$$
These moment-critical sets are $K_0$-stable and often do a good job at representing the $G_0$-orbits.  By Corollary \ref{c:J3}, if $\lambda\neq 1$, then the moment-critical subset for $J_{(\lambda)}$ is obtained by pushing forward the moment-critical subset for $J_3$ using the explicit intertwining map $T_{\nu}$.

These two ideas come together in the following way.  We take the usual action (with complex structure $J_1$) of $G_0$ on $\O$ as the fundamental topic of interest.  The moment-critical subset represents the topologically closed $G_0$-orbits---a typical outcome for actions on an affine variety \cite{KN}.  However, both the $G_0$-orbit structure and moment-critical sets for the action of $G_0$ on $T^*\left(G/Q\right)$ (with complex structure $J_3$) are quite different than for $J_1$.  Then for intermediate complex structures $J_{(\lambda)}$ ($0<|\lambda|<1$), the $G_0$-orbit structure resembles that for $J_1$, and the moment-critical subsets resemble those for $J_3$.  The new window on the action of $G_0$ on $\O$ comes from considering these $J_{(\lambda)}$.  Specifically, we use $T_{\nu}$ to push forward the moment-critical subset for $J_3$, obtaining the moment-critical subset for $J_{(\lambda)}$.  Here, since the holomorphic action of $G_0$ is just a camouflaged version of the action on $\O$, we push forward the $J_{(\lambda)}$-moment-critical subset via $T_\lambda^{-1}$ to obtain a {\it new} moment-critical subset for the usual action of $G_0$ on $\O$.  The remarkable fact is that one obtains a quite different moment-critical subset than for $J_1$, that in particular, represents different $G_0$-orbits.  In particular, it represents at least some of the non-closed orbits.  

Below, we show how this works out for $SL(2)$.  We will borrow freely from some of the (complicated) results in \cite{Br}, where the reader can find justifications of some of the background outlined below. 

In this example, we have $G_u=SU(2)$, $G_0=SU(1,1)$, and $K_0=\left\{\left(\begin{matrix}e^{i\theta}&0\\0 &e^{-i\theta}\end{matrix}\right):\theta\in\R\right\}$.  The $K_0$-orbit of each point in $\O$ is represented by an element of the form
\begin{equation}\Ad_{\left(\begin{matrix}\cos s&-\sin s\\\sin s&\cos s\end{matrix}\right)\left(\begin{matrix}e^{ic}&0\\0&e^{-ic}\end{matrix}\right)\left(\begin{matrix}\cosh r&\sinh r\\\sinh r&\cosh r\end{matrix}\right)}\Upsilon  \tag{*}\end{equation}
which under the identification of $\O$ with $\twist GQ{\m^-}$, corresponds to the point
\begin{equation}\left[\left(\begin{matrix}\cos s&-\sin s\\\sin s&\cos s\end{matrix}\right),\dfrac 1{2i}\left(\begin{matrix}0&0\\e^{-2ic}\cdot \sinh 2r&0\end{matrix}\right)\right]\tag{**}\end{equation}
We regard $\twist GQ{\m^-}$ as the cotangent bundle to the 2-sphere, with the base being $G/Q\simeq S^2$ and the fiber over the base point being $\m^-\simeq \R^2$.  In the version $(*)$, $s$ determines latitude on the 2-sphere (with $s=0$ giving the north pole, $s=\pi/4$ giving a point on the equator, and $s=\pi/2$ giving the south pole).  The parameter $r$ reflects magnitude in the fiber, while a change of $\theta$ produces rotation in the fiber.  Under this correspondence, the north pole is obtained when $s=0=r$ and corresponds to $\Upsilon\in\O$, and the south pole is obtained when $s=\pi/2$ and $r=0$, and this corresponds to $-\Upsilon$.  More generally, the correspondence $(**)$ carries $\O_0$ (resp. $-\O_0$) to the cotangent plane at the north (resp. south) pole, and $\O_u$ to the base 2-sphere.  The $K_0$-conjugacy class of $\frac i2\left(\begin{matrix}0&1\\1&0\end{matrix}\right)$ gives the equator of the 2-sphere; in $(*)$, it is achieved by $s=\pi/4$ and $r=0$.  

Below, we discuss the $G_0$-orbit structure and moment-critical subsets for the complex structures $J_3$ and $J_1$, then display the push-forward of the $J_3$-moment-critical subsets by $T_\lambda^{-1}\circ T_\nu$. 

\subsection{The complex structure $J_3$}

\subsubsection{$G_0$-orbit structure}  We interpret  in the context of the cotangent bundle.  
The group $G_0$ acts on the base with three orbits:  the northern hemisphere,  the southern hemisphere, and the equator.  These are represented by $\Upsilon$, $-\Upsilon$, and the point as in $(*)$ with $s=\pi/4$ and $c,r=0$.  As homogeneous spaces, the first two are isomorphic to $G_0/K_0$ and the third is $G_0/B_0$, where $B_0$ consists of the real points of a split Borel subgroup.  The closure of either of the hemispheres contains the equator.  

Any other $G_0$-orbit in $\twist GQ{\m^-}$ projects under $\twist GQ{\m^-}\rightarrow G/Q$ to one of the three orbits above. 

There is a ray of orbits projecting to the upper hemisphere, each of them intersecting the cotangent space at the north pole in a single $K_0$-orbit (a circle).  Each orbit is represented by a point of the form $(*)$ with $s=0=c$, and a choice of $r>0$ that uniquely determines the orbit.  Each such orbit is closed and is isomorphic to $G_0/\{\pm I\}$ as a homogeneous space.  There is a similar situation for orbits over the southern hemisphere.  

Finally, there are orbits that project to the equator.  There is a circle of such orbits, each isomorphic to $G_0/U_0$, where $U_0$ is the group of real points of the unipotent radical of a split Borel subgroup, extended by $\pm I$.  Each orbit is represented by a point as in $(*)$ with $s=\pi/4$, $r=1$ (or any positive number), and a choice of $\theta\in[0,\pi)$ that uniquely determines the $G_0$-orbit.  Each of these orbits contains the equator in its closure.  

\subsubsection{The moment-critical subset}\label{sss:J3critical}
The $K_0$-orbits of the moment-critical subsets are represented by:

$\bullet$  Points $(*)$ with $s=0$, $c=0$, and $r\geq 0$.  The $K$-orbits of these points comprise the cotangent space at the north pole.  Thus, each $G_0$-orbit in $\twist GQ{\m^-}$ projecting to the upper hemisphere intersects the moment-critical subset in a single $K_0$-orbit, namely, the points in the orbit that lie above the north pole.  This applies both to the closed $G_0$-orbits ($r>0$), and the upper hemisphere itself ($r=0$).  
There are similar considerations for the southern hemisphere, using $s=\pi/2$. 

$\bullet$ The point $(*)$ with $s=\pi/4$ and $r=0=c$.  The $K_0$-orbit of this point is exactly the equator of the 2-sphere.  Thus the equator is a $G_0$-orbit but consists entirely of moment-critical points; the non-closed $G_0$-orbits projecting to the equator are not represented by moment-critical points.  

\subsection{The complex structure $J_1$}

\subsubsection{$G_0$-orbit structure}  \label{sss:J1orbits}
The set $\O$ contains one-parameter families of orbits isomorphic to $G_0/\{\pm I\}$---all but four of them topologically closed, along with two orbits isomorphic to $G_0/K_0$, and one isomorphic to $G_0/T_0$, where $T_0$ is the group of real points of a split torus.  To explain how these orbits fit together, it's useful to consider a $G_0$-invariant function on $\O$.  We have that $\O=\left\{\frac i2\left(\begin{matrix}A&B\\C&-A\end{matrix}\right):A,B,C\in\C\text{ and } A^2+BC=1\right\}$; the real-valued function $f=\dfrac{2|A|^2-|B|^2-|C|^2}2$ is $G_0$-invariant and has range $(-\infty,1]$.  

The exceptional $G_0$-orbits occur when $f=\pm 1$.  There are two orbits on which $f=1$, namely the orbits $\pm \O_0$ of $\pm\Upsilon$, respectively, which are closed in $\O$ and isomorphic to $G_0/K_0$.  Given a point 
$\frac i2\left(\begin{matrix}A&B\\C&-A\end{matrix}\right)$ at which $f=1$,  one can determine whether it is in the $G_0$-orbit of $\Upsilon$ or $-\Upsilon$ according to whether the real part of $A$ is positive or negative.  

There are five orbits on which $f=-1$.  One of them is topologically closed, namely, the orbit of $\frac i2\left(\begin{matrix}0&1\\1&0\end{matrix}\right)$, which is the orbit isomorphic to $G_0/T_0$.    Four orbits contain this orbit in their closures; they are isomorphic to $G_0/\{\pm I\}$.  They are represented  by $\frac i2\begin{pmatrix}1&2\\0&-1\end{pmatrix}$, $\frac i2\begin{pmatrix}-1&2\\0&1\end{pmatrix}$, $\frac i2\begin{pmatrix} 1&0\\2&-1\end{pmatrix}$, and $\frac i2\begin{pmatrix}-1&0\\2&1\end{pmatrix}$.   Given that $f=-1$, the closed orbit is determined by the conditions $Re(A)=0$ and $|B|=|C|$.  The other four orbits are separated by the signs of $Re(A)$ and $|B|-|C|$.  

There are two orbits for each value of $f$ with   $-1<f<1$, each topologically closed and isomorphic to 
$G_0/\{\pm I\}$.  The sign of $\Re A$ separates the two orbits.  An interesting choice of representatives occurs when $C=B$ and $A,B\in\R^*$.  We obtain
matrices $\frac i2\begin{pmatrix}A&B\\B&-A\end{pmatrix}$ with $A^2+B^2=1$, on which $f$ takes the value $2A^2-1\in(-1,1)$.  

Similarly, when $f<-1$, there are two $G_0$-orbits for each choice of the invariant $f$, each closed and isomorphic to $G_0/\{\pm 1\}$. The sign of $|B|-|C|$ separates the two orbits.   Interesting representatives occur with the choice $A=0$, $B>0$, $B\neq 1$.  We obtain matrices $\frac i2\left(\begin{matrix}0&B\\B^{-1}&0\end{matrix}\right)$, on which $f$ takes the value $-(B^2+B^{-2})/2\in(-\infty, -1)$.  

\subsubsection{The moment-critical subset}\label{sss:J1critical}

$\phantom{.}$ 

$\bullet$  One collection of critical points comes from points $(*)$ with $r=0$.  Viewed in $\twist GQ{\m^-}$, this constitutes the base $2$-sphere; viewed in $\O$, it is  $\O_u$.     The point $(*)$ is 
$$\Ad_{\left(\begin{matrix}\cos s&-\sin s\\\sin s&\cos s\end{matrix}\right)}\Upsilon=\frac i2\left(\begin{matrix}\cos 2s&\sin 2s\\\sin 2s&-\cos 2s\end{matrix}\right)$$  which maps under $f$ to $\cos^2 2s-\sin^2 2s=\cos 4s$.  
We obtain all matrices of this form by a choice of $s\in [0,\pi)$.  The choices $s=0$ and $s=\pi/2$ produce $\pm\Upsilon$, at which $f=1$, and the choice $s=\pi/4$ produces $\frac i2\left(\begin{matrix}0&1\\1&0\end{matrix}\right)$, which lies on the closed orbit with $f=-1$.  The choices $s$ and $\pi/2+s$ (with $s\in(0,\pi/4)$) produce representatives of the two orbits at which $f=\cos 4s\in(-1,1)$.

$\bullet$ A second collection of critical points (overlapping with the first) comes from points $(*)$ with $s=\pi/4$, $c=\pi/2$, and $r$ arbitrary.  Together, these can be viewed as a line bundle over the equator in $\twist GQ{\m^-}$, or as the $K$-conjugacy class of $\frac i2\left(\begin{matrix}0&1\\1&0\end{matrix}\right)$ in $\O$.  The point $(*)$ is 
$$\Ad_{\frac 1{\sqrt 2}\left(\begin{matrix}1&-1\\1&1\end{matrix}\right)\left(\begin{matrix}\cosh r&\sinh r\\\sinh r&\cosh r\end{matrix}\right)  }\Upsilon=\frac i2\left(\begin{matrix}0&e^{-2r}\\e^{2r}&0\end{matrix}\right)$$
which maps under $f$ to $\cosh 4r$.  

Thus, we see that the critical sets comprise  the representatives of the closed orbits described in \ref{sss:J1orbits}.

\subsection{Alternate moment-critical points for the action of $G_0$ on $\O$} 
Fix $\lambda$ with $0<|\lambda|<1$ and take $\nu$ as in Section \ref{s:deformations}.  Using Theorem \ref{thm:Tlambdaformula}, we can compute $T_\lambda^{-1}\circ T_\nu$ for $SL(2)$.  We consider the moment-critical subsets for $J_3$ (\S\ref{sss:J3critical}) and report their images under $T_\lambda^{-1}\circ T_\nu$.

$\bullet$ Regarding the {\em first} type,  the moment-critical point (cotangent vector at the north pole) 
$$\Ad_{\left(\begin{matrix} e^{ic}&0\\0&e^{-ic}\end{matrix}\right)\left(\begin{matrix}\cosh r&\sinh r\\\sinh r&\cosh r\end{matrix}\right)    }\Upsilon=\frac i2\left(\begin{matrix}\cosh 2r&e^{2ic}\sinh 2r\\-e^{-2ic}\sinh 2r&-\cosh 2r\end{matrix}\right)$$
pushes forward under $T_\lambda^{-1}\circ T_\nu$ to 

$$\begin{aligned}&\Ad_{\left(\begin{matrix}e^{ic}&0\\0&e^{-ic}\end{matrix}\right)\left(\begin{matrix}\cos s'&\sin s'\\-\sin s'&\cos s'\end{matrix}\right)\left(\begin{matrix}\cosh r'&\sinh r'\\\sinh r'&\cosh r'\end{matrix}\right)}\Upsilon\\
\phantom{.}\qquad&=\frac i2\left(  \begin{matrix}
\cos 2s'\cosh 2r'&e^{2ic}(\sinh 2r'-\sin 2s'\cosh 2r')\\
-e^{-2ic}(\sinh 2r'+\sin 2s'\cosh 2r')&-\cos 2s'\cosh 2r'
\end{matrix}\right)\end{aligned}$$
where 
$$\begin{aligned}
s'&=\frac 12\arctan\left( \dfrac{-(\sinh 2r)\sqrt{(\cosh^2 2r)\left(1-|\lambda|^2\right)^2+4|\lambda|^2}}{2|\lambda|}\right)\\
r'&=\frac 12\text{arcsinh}\left(\frac 12\left(\frac 1{|\lambda|}-|\lambda|\right)\sinh 2r\right)
\end{aligned}$$
are odd functions of $r$.  

\noindent One computes that the invariant $f$ takes the value 
$$f=1-2\sin^2(2s')\cosh^2(2r')=  1-\tanh^2(2r)\left(4|\lambda|^2+(1-|\lambda|^2)^2\cosh^2(2r)\right)/{2|\lambda|^2} $$
from which it follows that $f$ takes on all values in $(-\infty,1]$ on this critical set.  
The function $f$ takes on the value $-1$ in this critical set at points $\frac i2\begin{pmatrix}A&B\\C&-A\end{pmatrix}$ where we see that $A$ is real and positive, hence at non-closed orbits.  

$\bullet$ The {\em second} type, consisting of points on the equator, are fixed by $T_\lambda^{-1}\circ T_\nu$. These are the ``overlap'' in the two critical sets for $J_1$, as reported in \S\ref{sss:J1critical}. 

Thus, we obtain a critical subset which, on the set $f=-1$, represents the closed orbit (via critical points of the second type), and two of the four non-closed orbits (via critical points of the first type).  This is remarkably different from what was described for $J_1$ in Section \ref{sss:J1critical}.


\end{document}